\documentclass[12pt]{amsart}
\usepackage{url}
\usepackage{graphicx,amsmath,amssymb,url,pgf, tikz}
\usepackage{ytableau}
\usepackage[all]{xy}
\usepackage{enumitem}
\numberwithin{equation}{section}
\newtheorem{prop}{Proposition}[section]
\newtheorem{thm}[prop]{Theorem}
\newtheorem{cor}[prop]{Corollary}
\newtheorem{lem}[prop]{Lemma}

 \theoremstyle{definition}
 \newtheorem{defn}[prop]{Definition} 

 \theoremstyle{definition}
 \newtheorem{remark}[prop]{Remark} 

\newcommand{\given}{\ : \ }  

\newcommand{\pic}{\begin{tikzpicture}}
\newcommand{\epic}{\end{tikzpicture}}
\newcommand{\Inv}{\mbox{Inv}}

\newcommand{\ma}{\mathcal{A}}

\newcommand{\mb}{\mathcal{B}}
\newcommand{\mc}{\mathcal{C}}
\newcommand{\md}{\mathcal{D}}
\newcommand{\mg}{\mathcal{G}}
\newcommand{\msg}{\mathcal{SG}}

\newlength{\cellsize} 
\newsavebox{\cell}
\sbox{\cell}{\begin{picture}(18,18)
\put(0,0){\line(1,0){18}}
\put(0,0){\line(0,1){18}}
\put(18,0){\line(0,1){18}}
\put(0,18){\line(1,0){18}}
\end{picture}}
\newcommand\cellify[1]{\def\thearg{#1}\def\nothing{}%
\ifx\thearg\nothing
\vrule width0pt height\cellsize depth0pt\else
\hbox to 0pt{\usebox{\cell} \hss}\fi%
\vbox to \cellsize{
\vss
\hbox to \cellsize{\hss$#1$\hss}
\vss}}
\newcommand\tableau[1]{\vtop{\let\\\cr
\baselineskip -16000pt \lineskiplimit 16000pt \lineskip 0pt
\ialign{&\cellify{##}\cr#1\crcr}}}
\begin{document}

\title[Signed Little Map]{Coxeter-Knuth graphs and a signed Little map\\
  for type $B$ reduced words}
\date{\today}
\author[S. Billey, Z. Hamaker, A. Roberts and B. Young]{Sara Billey, Zachary Hamaker, Austin Roberts \and Benjamin Young}
\thanks{Billey and Roberts were partially supported by grant DMS-1101017 from
the NSF}
\address{SB and AR: Department of Mathematics, Box 354350, University of Washington, Seattle, WA 98195, USA}
\address{ZH: Department of Mathematics, Hinman Box 6188, Dartmouth College, Hanover, NH 03755, USA} 
\address{BY: Department of Mathematics, 1222 University of Oregon, Eugene, OR 97403} 
\email{billey@math.washington.edu}
\email{zachary.hamaker@gmail.com}
\email{austinis@math.washington.edu}
\email{bjy@uoregon.edu}
\maketitle
\begin{abstract}
We define an analog of David Little's algorithm for reduced words in type B, and investigate its main properties.  In particular, we show that our algorithm preserves the recording tableaux of Kra\'{s}kiewicz insertion, and that it provides a bijective realization of the type B transition equations in Schubert calculus. Many other aspects of type A theory carry over to this new setting.  Our primary tool is a shifted version of the dual equivalence graphs defined by Assaf and further developed by Roberts.  We provide an axiomatic characterization of shifted dual equivalence graphs, and use them to prove a structure theorem for the graph of type B Coxeter-Knuth relations.
\end{abstract}


\section{Introduction}\label{s:intro}


Stanley symmetric functions $F_w$ appear in the study of reduced words
of permutations \cite{stanleysymm}, the representation theory of
generalized Specht modules \cite{kraskiewicz}, and the geometry of
positroid varieties \cite{knutson.lam.speyer}.  The $F_w$ are known to
have a Schur positive expansion with coefficients determined by the
Edelman-Greene correspondence.  This correspondence associates to each
reduced word a pair of tableaux $(P,Q)$ of the same shape where the
second tableau is standard. These symmetric functions $F_w$ can be
defined as the sum of certain fundamental quasisymmetric functions
where the sum is over all reduced words for $w\in S_n$, denoted
$R(w)$.  In particular, the coefficient of $x_1 x_2 \cdots x_{\ell(w)}
$ in $F_w$ equals $|R(w)|$.  There is a recurrence relation for $F_w$
derived from Lascoux and Sch\"utzenberger's transition equation for
Schubert polynomials \cite{lascouxschutzenbergerschubert} of the form
$$
F_w = \sum_{w' \in T(w)} F_{w'},
$$
along with the base cases that $F_w$ is a single Schur function if $w$
has at most one descent; in this case we say $w$ is
\textit{Grassmannian}.  By taking the coefficient of $x_1 x_2 \cdots x_{\ell(w)}
$ on both sides of the
recurrence, we see that the sets $R(w)$ and $\cup_{w' \in T(w)} R(w')$
are equinumerous.

David Little gave a remarkable bijection between $R(w)$ and $\cup_{w'
\in T(w)} R(w')$ \cite{little2003combinatorial} inspired by the
lectures of Adriano Garsia, which are published as a book
\cite{garsia2002saga} .  This algorithm is a finite sequence of steps, 
each of which decrements one letter in the word. If ever a 1 is
decremented to a 0, then instead the whole reduced word is lifted up
by one to make space for one extra generator.  This bijection is an
instance of a more general phenomenon known as \emph{Little bumps}.  


Recently, Hamaker and Young \cite{hamaker2012relating} have shown that
Little bumps preserve the 
recording tableaux under the Edelman-Greene
correspondence. This proved a conjecture of Thomas Lam
\cite[Conj. 2.5]{lam2010stanley}.  They further show that all reduced
words 
 with a given 
recording tableau $Q$ under the Edelman-Greene
correspondence are connected via Little bumps.  Edelman and Greene
gave a refinement on the Coxeter relations in type $A$, which they
call 
\textit{Coxeter-Knuth} relations.  These relations
preserve the 
insertion tableaux under the Edelman-Greene correspondence, and
the set of reduced words which have a fixed 
insertion tableau $P$ is connected by elementary
Coxeter-Knuth relations.  Hamaker and Young further showed that two
reduced words that differ by an elementary Coxeter-Knuth relation give
rise to $Q$ tableaux that differ in exactly two positions.  This can
be made more precise.  Consider the graph $CK_A(w)$ on all reduced
words for $w$ with an edge labeled $i$ between two reduced words
$a=a_1a_2\cdots a_p$ and $b=b_1b_2\cdots b_p$ whenever $a$ and $b$
differ by an elementary Coxeter-Knuth relation in positions
$i,i+1,i+2$.  Call $CK_A(w)$ a \textit{Coxeter-Knuth} graph.  Using
the theory of dual equivalence graphs due to Assaf
\cite{assaf2010dual} and the equivalent axioms given by Roberts
\cite{roberts2013dual}, one can easily show that $CK_{A}(w)$ is a dual
equivalence graph and the $Q$ tableaux for two reduced words differing
by an elementary Coxeter-Knuth move differ by one of Haiman's dual
equivalence moves \cite{haiman1992dual}.

In this paper, we define the analog of the Little bump
$B^{\delta}_{(i,j)}$ on reduced words for the signed permutations
$B_{n}$, and show that these maps satisfy many of the same
properties as in the original case.  The superscript $\delta \in \{+,-\}$ denotes the direction of the bump, and the subscript $(i,j)$ indicates the crossing where the bump begins.
 In particular, there is a close
connection to the Stanley symmetric functions for types B and C
defined in \cite{BH}, see also \cite{FK2,lam1995b}.  These Stanley
symmetric functions again satisfy a transition equation
\cite{billey1998transition}, which proves that $R(w)$ is equinumerous
with a certain union of $R(w')$'s.  

To concretely state our first main result, we need to establish some
notation.  A signed permutation $w \in B_{n}$ is a bijection from
$\{-n,\dots -1, 1,2,\dots , n \}$ to itself such that $w(i)=-w(-i)$.
One could represent $w$ in one-line notation either by listing
$[w(-n),w(-n+1),\dots,w(-1),w(1),\dots , w(n)]$ in long form or simply
$[w(1),\dots , w(n)]$ in short form.  For example,
$[1,\bar{2},\bar{4},3,\bar{3},4,2,\bar{1}]$ and
$[\bar{3},4,2,\bar{1}]$ represent the same element in $B_{4}$ where
$-i$ is denoted $\bar{i}$.  For our purposes, we identify $v\in B_{n}$
with the element $w\in B_{n+1}$ such that $v(i)=w(i)$ for $1\leq i\leq
n$ and $w(n+1)=n+1$.  Set $B_{\infty} = \cup B_{n}$ in this
identification.  For $i<j \in \mathbb{Z}\setminus \{0 \}$, let
$t_{ij}$ be the (signed) transposition such that $t_{ij}(i)=j$,
$t_{ij}(j)=i$, $t_{ij}(-i)=-j$, $t_{ij}(-j)=-i$ and for
every integer $k \not \in \{\pm i, \pm j,0 \}$ we have $t_{ij}(k)=k$.  If $w
\in B_{\infty}$ has $w(1)<w(2)<\dots $, we say $w$ is
\textit{increasing}.  If $w$ is not increasing, let $(r<s)$ be the
lexicographically largest pair of positive integers such that
$w_{r}>w_{s}$.  Set $v=wt_{rs}$.  Let $T(w)$ be the set of all signed
permutations $w'=vt_{ir}$ for $i<r, i \neq 0$ such that $\ell(w')=\ell(w)$.

\begin{thm}\label{t:little.bijection}
Using the notation above, if $w \in B_{\infty}$ is not increasing,
then the particular Little bump $B^{-}_{(r,s)} \ : \ R(w)
\longrightarrow \bigcup_{w'\in T(w)}R(w')$ is the bijection
predicted by the transition equation for type C Stanley symmetric
functions.
\end{thm}

The analog of Edelman-Greene insertion and elementary Coxeter-Knuth
relations for signed permutations were given by Kra\'skiewicz
\cite{kraskiewicz1989reduced}.  Kra\'skiewicz insertion inputs a
reduced word $a$ and outputs two shifted tableaux $(P'(a),Q'(a))$ of
the same shifted shape where the recording tableau $Q'(a)$ is
standard.  We develop some properties of the signed Little bumps and
the recording tableaux as maps on reduced words summarized in the next
theorem.

\begin{thm}\label{t:Q.preserving} Suppose $w$ and $wt_{ij}$ are signed
permutations such that $\ell(w)=\ell(w t_{ij}) + 1$.  
\begin{enumerate}
\item The Little bump $B^{\delta}_{(i,j)}$ maps
$R(w)$ to reduced words for some signed permutation $w'=wt_{ij}t_{kl}$ with $\ell(w) = \ell(w')$. 
\item Two reduced words $a$ and $b$ are connected via Little bumps if and
only if  $Q'(a)=Q'(b)$ under Kra\'skiewicz insertion.  
\item For each standard shifted tableau $Q'$, there exists a unique
reduced word $a$ for an increasing signed permutation such that $Q'(a)=Q'$.  
\end{enumerate}
\end{thm}


The Coxeter-Knuth relations given by Kra\'skiewicz lead to a type B
Coxeter-Knuth graph $CK_{B}(w)$ for each $w \in B_{\infty}$.  An
important step in proving Theorem~\ref{t:Q.preserving} is showing that
two reduced words for signed permutations that differ by an
elementary Coxeter-Knuth relation give rise to two $Q'$ tableaux that
differ by one of Haiman's shifted dual equivalence moves
\cite{haiman1992dual}.  In fact, shifted dual equivalence completely
determines the graph structure for type B Coxeter-Knuth graphs and
vice versa.  Thus, we define shifted dual equivalence graphs in
analogy with the work of Assaf and Roberts on dual equivalence graphs. 

\begin{thm}\label{t:ckg}
  Every type B Coxeter-Knuth graph $CK_{B}(w)$ is a shifted dual
  equivalence graph with signature function given via peak sets of
  reduced words.  The isomorphism is given by $Q'$ in Kra\'skiewicz
  insertion.  Conversely, every connected shifted dual equivalence
  graph is isomorphic to the Coxeter-Knuth graph for some increasing
  signed permutation.
\end{thm}

Putting Theorem~\ref{t:Q.preserving} and Theorem~\ref{t:ckg} together,
one can see that Little bumps in both type A and type B play a similar
role for Stanley symmetric functions as jeu de taquin plays in the
study of Littlewood-Richardson coefficients for skew-Schur functions.
In particular, let us say that words $a$, $a'$ \emph{communicate} if there is a sequence of Little bumps which transforms $a$ into $a'$. We will show that there is exactly one reduced word for a unique
increasing signed permutation in each communication class under Little
bumps.  

We give local axioms characterizing graphs isomorphic to shifted dual
equivalence graphs or equivalently Coxeter-Knuth graphs of type B.
We state the theorem here using some terminology that is developed in
Section~\ref{s:axioms}.

\begin{thm} \label{SDEG properties} A signed colored graph
  $\mathcal{G} = (V, \sigma, E)$ of shifted degree $[n]$ is a shifted
  dual equivalence graph if and only if the following local properties hold.
\begin{enumerate}
\item If $I$ is any interval of integers with $|I|\leq 9$, then each
  component of $\mg|_I$ is isomorphic to the standard shifted dual
  equivalence graph of a shifted shape of size up to $|I|$.
\item If $i,j \in \mathbb{N}$ with $|i-j|>3$, $(u, v) \in E_{i}$ and
  $(u,w) \in E_{j}$, then there exists a vertex $y \in V(\mg)$ such
  that $(v,y) \in E_{j}$ and $(w,y) \in E_{i}$.
\end{enumerate}
\end{thm}

We propose that the study of Coxeter-Knuth graphs initiated in this
paper is an interesting way to generalize dual equivalence graphs to
other Coxeter group types.  For example, in type $A$, dual equivalence
graphs have been shown to be related to crystal graphs
\cite{Assaf2008crystals}.  Furthermore, the transition equation due to
Lascoux and Sch\"utzenberger follows from Monk's formula for
multiplying a special Schubert class of codimension 1 with an
arbitrary Schubert class in the flag manifold of type $A$.  The
elementary Coxeter-Knuth relations could have been derived from the
Little bijection provided one understood the Coxeter-Knuth relations
for the base case of the transition equations in terms of Grassmannian
permutations.  The transition equations for the other classical groups
follow from Chevalley's generalization for Monk's formula on Schubert
classes \cite{CHEV}.  In fact, there is a very general Chevalley
Formula for all Kac-Moody groups \cite{lenart2012chev}.

We comment on one generalization which did not work as hoped.  In type
$A$, Chmutov showed that the molecules defined by Stembridge's axioms
can be given edge labels in such a way that the graphs are dual
equivalence graphs \cite{chmutov2013preprint}.  Alas, in type B,
this does not appear to be possible.  The Kazhdan-Lusztig graph for
$B_3$ has a connected component with an isomorphism type that does not
occur for dual equivalence graphs or shifted dual equivalence graphs.
Namely, the component of $[2,1,\bar{3}]$ is a tree with 4 vertices and
3 leaves $[1,\bar{2}, \bar{3}], [\bar{2},\bar{3},1],[2,1,\bar{3}]$.

The paper proceeds as follows.  In Section~\ref{s:background}, we
review the necessary background on permutations and signed
permutations as Coxeter groups.  In Section~\ref{s:signed.little}, we
formally define the signed Little bumps and pushes.  The key tool we
use to visualize the algorithms is the wiring diagram of a reduced
word. The conclusion of the proof of Theorem~\ref{t:little.bijection}
is given in Corollary~\ref{c:little.bijection}, and
Theorem~\ref{t:Q.preserving}(1) follows directly from
Theorem~\ref{thm:bump.image}.  The relationships between Little bumps, the
recording tableaux under the Kra\'{s}kiewicz insertion,
Coxeter-Knuth moves of type B and shifted dual equivalence moves are 
discussed in Section~\ref{s:Kraskiewicz}. The main results of this
section prove Theorem~\ref{t:Q.preserving}, parts (2) and (3).  In
Section~\ref{s:axioms}, the shifted dual equivalence graphs are
equivalently defined in terms of either shifted dual equivalence moves
or Coxeter-Knuth moves proving Theorem~\ref{t:ckg}.
Theorem~\ref{t:ckg} is an easy consequence of this definition and the
machinary built up in Sections~\ref{s:background}
and~\ref{s:Kraskiewicz}.  We go on to prove many lemmas leading up to
the axiomatization of shifted dual equivalence graphs proving
Theorem~\ref{SDEG properties}.  We conclude with some interesting open
problems in Section~\ref{s:open}.

We recently learned that Assaf has independently considered shifted
dual equivalence graphs in connection to a new Schur positive
expansion of the Schur $P$-polynomials \cite{assaf2014shifted}.  In
particular, the connection between shifted dual equivalence graphs and
Little bumps is new to this article.

\section{Background}\label{s:background}

Let $W$ be a Coxeter group with generators $S=\{s_1, \ldots , s_{n}
\}$ and elementary relations $\left(s_{i} s_{j} \right)^{m(i,j)}=1$.
For $w \in W$, let $\ell(w)$ be the minimal length of any expression
$s_{a_{1}}\cdots s_{a_{q}}=w$.  If $\ell(w)=p$, we say
$s_{a_{1}}\cdots s_{a_{p}}$ is a \textit{reduced expression} and the
list of subscripts $a_{1}a_2\cdots a_{p}$ is a \textit{reduced word}
for $w=s_{a_{1}}s_{a_2}\cdots s_{a_{p}}$.  Let $R(w)$ be the set of reduced
words for $w$.

For $w \in W$, one can define a graph $G(w)$ with vertices given by the
reduced words of $w$ using the Coxeter relations.  In this graph, any
two reduced words are connected by an edge if they differ only by an
elementary relation of the form $s_{i} s_{j} s_{i}\cdots = s_{j} s_{i}
s_{j} \cdots$ where each side is a product of $m(i,j)$ generators.  It
is a well known theorem, sometimes attributed to Tits, that this graph
is connected \cite[Thm. 3.3.1]{b-b}.

\subsection{Type A}

The symmetric group $S_n$ is the Coxeter group of type $A_{n-1}$.
For our purposes, we can think of $w \in S_n$ in one-line notation as
$w=[w_{1},w_{2},\ldots, w_{n} ]$ or as $w=[w_{1},w_{2},\dots ] \in
S_\infty$ with $w_i=i$ for all $i>n$.  Let $t_{ij}$ be the
transposition interchanging $i$ and $j$ and fixing all other values.
Then right multiplication by $t_{ij}$ interchanges the values in
positions $i$ and $j$ in $w$.  

The group $S_n$ is minimally generated by the \textit{adjacent
  transpositions} $s_1,\ldots, s_{n-1}$, where $s_i =t_{i,i+1}$, with
elementary Coxeter relations
\begin{enumerate}
\item \textbf{Commutation}:  $s_i s_j =s_j s_i$ provided $|i-j|>1$,  
\item \textbf{Braid}: $s_i s_{i+1} s_i = s_{i+1} s_i s_{i+1}$.
\end{enumerate}
For example, if $w=[2,1,5,4,3]$, then 
$$R(w) = \{1343,3143,3413, 3431,4341, 4314, 4134, 1434 \}$$ 
and $G(w)$ is a cycle on these 8 vertices.

In \cite{edelman1987balanced}, Edelman-Greene (EG) gave an insertion
algorithm much like the famous Robinson-Schensted-Knuth (RSK)
algorithm for inserting a reduced word into a tableau for some
partition $\lambda=(\lambda_1\geq \lambda_2 \geq \ldots \geq
\lambda_k>0)$.  The one difference in EG insertion is that when
inserting an $i$ into a row that already contains an $i$ and $i+1$,
we skip that row and insert $i+1$ into the next row.  If one keeps track
of the recording tableau of the insertion, then the process is
invertible.  Let $P(a)$ be the EG insertion tableau for $a=a_{1}\dots
a_{p}$ and let $Q(a)$ be the recording tableau.  Define $EG(w) = \{P:
P=P(a)\ \text{ for some } a \in R(w) \}$.

For example, using EG insertion, the reduced word $1343$ inserts to
give
\ytableausetup{aligntableaux=bottom, smalltableaux}
\[
\ytableaushort{1} \raisebox{.6ex}{ $\to$ }
\ytableaushort{13}  \raisebox{.6ex}{ $\to$ }
\ytableaushort{134}  \raisebox{.6ex}{ $\to$ }
\ytableaushort{4,134}.
\]
So 
\[
\ytableausetup{aligntableaux=center, smalltableaux}
P(1343)= \ytableaushort{4,134} 
\hspace{.3in} \text{and} \hspace{.3in}
Q(1343)=\ytableaushort{4,123}  
\]
in French notation.

The EG recording tableau $Q(a)$ is a \textit{standard Young tableau}
of partition shape $\lambda$, denoted $SYT(\lambda)$.  These are
bijective fillings of the Ferrers diagram for the partition $\lambda$
with rows and columns increasing.  The \textit{row reading word} of a
standard tableau $T$ is the permutation in one-line notation obtained
by reading along the rows of $T$ in the French way, left to right and
top to bottom.  The \textit{ascent set} of $T$ is the set of all $i$
such that $i$ precedes $i+1$ in the row reading word of $T$.
Similarly, define the \textit{ascent set} of a reduced word
$a=a_1\cdots a_p$ to be $\{j: a_j<a_{j+1}\}$.  If a position is not an
ascent, it is called a \textit{descent}.

\begin{thm}\cite[Theorems 6.25 and 6.27]{edelman1987balanced}
  Fix $w \in S_{\infty}$ and $P \in EG(w)$.  Then, the recording
  tableau for EG insertion gives a bijection between $\{a\in R(w):
  P(a)=P \}$ and the set of standard Young tableaux of the same shape
  as $P$.  Furthermore, this bijection preserves ascent sets.
\end{thm}

\begin{defn}\label{d:EG.coeffs}
Let $a_{\lambda,w }$ be the number of distinct tableaux $P \in EG(w)$
such that $P$ has shape $\lambda$.  We call these numbers the
\textit{Edelman-Greene} coefficients.    
\end{defn}

\begin{defn}\label{d:stanley.an}\cite{stanleysymm}
For $w \in S_{\infty}$ and $a =a_{1}\cdots a_{p}\in R(w)$, let $I(a)$
be the set of all increasing integer sequences $1 \leq i_{1}\leq
i_{2}\leq \dotsb \leq i_{p}$ such that $i_{j}< i_{j+1}$ whenever
$a_{j}<a_{j+1}$.  The \textit{Stanley symmetric function} $F_{w}= F^{A}_{w}$ is
defined by 
\[
F^{A}_{w} = \sum_{a \in R(w)} \sum_{i_{1}i_{2}\dotsb i_{p}\in I(a)} x_{i_{1}} x_{i_{2}}\cdots x_{i_{p}}.
\]
\end{defn}
Here, the inner summation $\sum_{i_{1}i_{2}\dotsb i_{p}\in I(a)}
x_{i_{1}} x_{i_{2}}\cdots x_{i_{p}}$ is the fundamental quasisymmetric
function indexed by the ascent set of $a$
\cite[Ch. 7.19]{stanley2001enumerative}.  Edelman-Greene showed that
the ascent set of $a \in R(w)$ agrees with the ascent set of $Q(a)$.
Furthermore, Ira Gessel \cite{gessel1984} showed that the Schur
function $s_{\lambda}$ is the sum over all standard tableaux $T$ of
shape $\lambda$ of the fundamental quasisymmetric function by the
ascent set of $T$.  Putting this together gives the following theorem.

\begin{thm}\cite[Theorem 6.27]{edelman1987balanced}
Fix $w \in S_{\infty}$.  Then 
\[
F^A_{w}  = \sum_{\lambda} a_{\lambda', w} s_{\lambda},
\]
where $s_{\lambda}$ is the Schur function indexed by the partition
$\lambda$, $\lambda'$ is the conjugate partition obtained from
$\lambda$ by counting the length of the columns in the Ferrers
diagram, and each $a_{\lambda', w}$ is a nonnegative integer given in
Definition~\ref{d:EG.coeffs}.
\end{thm}

Edelman-Greene also characterized when two reduced expressions give rise
to the same $P$ tableau by restricting the elementary Coxeter relations.
For this characterization, they define the \textit{elementary
Coxeter-Knuth relations} to be either a braid move or a witnessed
commutation move:
\begin{enumerate}
\item $i k j \leftrightarrow k i j $ for all $i<j<k$,
\item $j i k  \leftrightarrow j k i  $ for all $i<j<k$,
\item $i (i+1) i \leftrightarrow (i+1) i (i+1) $.
\end{enumerate}
Two words which are connected via a sequence of Coxeter-Knuth relations
are said to be in the same \textit{Coxeter-Knuth class}. 

\begin{thm}\cite[Theorem 6.24]{edelman1987balanced}\label{thm:EG.CK}
Let $a,b\in R(w)$.  Then $P(a)=P(b)$ if and only if $a$ and $b$ are in
the same Coxeter-Knuth class.   
\end{thm}

In the example $w=[2,1,5,4,3]$, there are three Coxeter-Knuth classes
$\{3143,3413 \}$,
 $\{3431, 4341, 4314 \}$, and  $\{1343, 4134, 1434 \}$,  which respectively insert
to the three $P$ tableaux: 
\[
\ytableausetup{aligntableaux=bottom, smalltableaux}
\ytableaushort{34,13} \hspace{.2in}
\ytableaushort{4,3,14}\hspace{.2in}
\ytableaushort{4,134}.
\]

Let $CK_A(w)$ be the \textit{Coxeter-Knuth graph} for $w \in S_\infty$
with vertices $R(w)$ and colored (labeled) edges constructed using
Coxeter-Knuth relations.  An edge between $a$ and $b$ is labeled $i$
if $a_{j}=b_{j}$ for all $j \not \in \{i,i+1,i+2 \}$ and
$a_{i}a_{i+1}a_{i+2}$ and $b_{i}b_{i+1}b_{i+2}$ differ by an
elementary Coxeter-Knuth relation.  To each vertex $a \in R(w)$
associate a signature determined by its ascent set, $\sigma(a) = \{j:
a_j<a_{j+1}\}$.  If $\ell(w)=p$, we denote a subset $S\subset \{1,\ldots,
p-1\}$ by a sequence in $\{+,-\}^{p-1}$ where $+$ in the $j^{th}$
position means $j \in S$. Here $\sigma_j(a) =+$ if $a_j<a_{j+1}$ and
$\sigma_j(a) =-$ if $a_j>a_{j+1}$.  See Figure~\ref{fig:21543} for an example and compare to $G(21543)$, which is a cycle with eight vertices, as mentioned above.

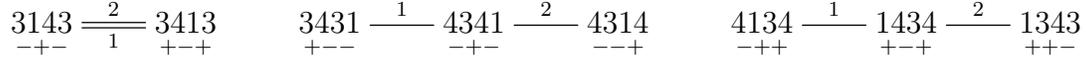
\begin{figure}[ht]
\ytableausetup{aligntableaux=center, smalltableaux}
\[
 \xymatrix{
 \underset{-+-}{3143}\ar@{=}[r]^{2}_{1}&  \underset{+-+}{3413}
 &
  \underset{+--}{3431} \ar@{-}[r]^{1} &  \underset{-+-}{4341} \ar@{-}[r]^{2} &  \underset{--+}{4314} 
&
  \underset{-++}{4134} \ar@{-}[r]^{1} &  \underset{+-+}{1434} \ar@{-}[r]^{2} & \underset{++-}{1343} 
 }
 \]
  \caption{\label{fig:21543} The Coxeter-Knuth graph of $w=[2,1,5,4,3]$.}
\end{figure}


Let $\alpha_i$ be the involution defined by the edges of $CK_A(w)$,
namely $\alpha_i(a)=b$ provided $a$ and $b$ are connected by an edge
colored $i$, or equivalently an elementary Coxeter-Knuth relation on
positions $i,i+1,i+2$.  If $a$ is not contained in an $i$-edge, then
define $\alpha_i(a) =a$.

The type $A$ Coxeter-Knuth graphs are closely related to dual
equivalence graphs on standard tableaux as defined by Assaf
\cite{assaf2010dual}.  For a partition $\lambda$, one defines a
\textit{standard dual equivalence graph} $\mg_{\lambda}$ to be the
graph with vertex set given by $SYT(\lambda )$, and an edge colored
$i$ between any two tableaux that differ by an elementary dual
equivalence defined as follows.

\begin{defn}\cite{haiman1992dual}
  Given a permutation $\pi\in S_n$, define the \emph{elementary dual
    equivalence operator} $d_i$ for all $1\leq i\leq n-2$ as follows.
  Say $\{i,i+1,i+2\}$ occur in positions $a<b<c$ in $\pi$, then
  $d_i(\pi)=\pi t_{ac}$ provided $\pi(b) \neq i+1$ and $d_i(\pi)=\pi$
  otherwise.  Dual equivalence operators also act on standard tableaux
  by acting on their row reading word. 
\end{defn}

It was observed by the first author that the following theorem holds
by combining the work in the original version of
\cite{hamaker2012relating} and \cite{roberts2013dual}.  This was the
start of our collaboration.

\begin{thm}\cite[Thm.  1.3]{hamaker2012relating}
  The graph $CK_{A}(w)$ is isomorphic to a disjoint union of standard
  dual equivalence graphs for each $w \in S_{n}$. The isomorphism
  preserves ascent sets on vertices.  On each connected component, the
  Edelman-Green $Q$ function provides the necessary isomorphism.
  Furthermore, ascent sets are preserved.
\end{thm}


\subsection{Type B/C}

The hyperoctahedral group, or signed permutation group $B_n$ is also a
finite Coxeter group.  This group is the Weyl group of both the root
systems of types B and C of rank $n$.  Recall from
Section~\ref{s:intro} that we have defined the \textit{(signed)
  transposition} $t_{ij}$ to be the signed permutation interchanging
$i$ with $j$ and $-i$ with $-j$ for all $i,j\neq 0$.  The group $B_n$
is generated as a Coxeter group by the adjacent transpositions $s_1,
\ldots, s_{n-1}$ with $s_i =t_{i,i+1}$ plus an additional generator
$s_0=t_{-1,1}$. Thus, if $w=[w_{1},\dots , w_{n}] \in B_n$, then
$ws_0=[-w_1,w_2,\ldots,w_n]$.  For example,
$[\bar{3},2,1]=s_1s_2s_1s_0 \in B_3$.  Again, let $R(w)$ denote the
set of all reduced words for $w$.  Note that if $w \in S_{n}$, then it
can also be considered as an element in $B_{n}$ with the same reduced
words.  The \textit{elementary relations} on the generators are given
by
\begin{enumerate}
\item \textbf{Commutation}:  $s_i s_j =s_j s_i$ provided $|i-j|>1$,  
\item \textbf{Short Braid}: $s_i s_{i+1} s_i = s_{i+1} s_i s_{i+1}$ for all $i>0$, 
\item \textbf{Long Braid}: $s_0 s_{1} s_0 s_1 = s_{1} s_0 s_{1}s_0$.
\end{enumerate}

Shifted tableaux play the same role in types B/C as the usual
tableaux play in type $A$.  Given a strict partition $\lambda =
(\lambda_{1}>\lambda_{2}>\dotsb >\lambda_{k}>0)$, the \textit{shifted
  shape} $\lambda$ is the set of squares in positions $\{(i,j) \given
1 \leq i\leq k, i\leq j\leq \lambda_{i}+i-1 \}$.  A \textit{standard
  shifted tableau} $T$ is a bijective filling of a shifted shape with
positive integers with rows and columns increasing. For example, see
$T$ in Figure~\ref{fig: tableaux}.  Let $SST(\lambda)$ be the set of
standard shifted tableaux of shifted shape $\lambda$.

\begin{figure}[h]
 \ytableausetup{aligntableaux=bottom}
S=
\begin{ytableau}
7\\
5&6&8\\
1&2&3&4&9
\end{ytableau},   \hspace{.3in}
 T=
\begin{ytableau}
\none&\none&7\\
\none&3&6&8\\
1&2&4&5&9
\end{ytableau},    \hspace{.3in}
U = \begin{ytableau}
\none&\none&0\\
\none&2&1&0\\
3&2&1&0&1
\end{ytableau}
\caption{A standard tableau $S$, standard shifted tableau $T$ and unimodal
  tableau $U$ respectively of shape $\lambda=(5,3,1)$.
\label{fig: tableaux}
}
\end{figure}

We will also need to consider another type of tableaux on shifted
shapes.  We say a list $r = r_1 \dots r_l$ is \emph{unimodal} if there
exists an index $j$, referred to as the \textit{middle}, such that
$r_1 \dots r_j$ is decreasing and $r_j \dots r_l$ is increasing.  A
\textit{unimodal tableau} $T$ is a filling of a shifted shape with
nonnegative integers such that the reading word along each row is
unimodal. 

In 1989, Kra\'skiewicz\cite{kraskiewicz1989reduced} gave an analog of
Edelman-Greene insertion for reduced words of signed permutations.
Kra\'skiewicz insertion is a variant of the mixed shifted insertion
of~\cite{haiman1989mixed} that maps a reduced word $b$ of a signed
permutation to the pair of shifted tableaux $(P'(b),Q'(b))$ where
$Q'(b)$ is a standard shifted tableau and $P'(b)$ is a unimodal
tableau of the same shape such that the reading word given by reading
rows left to right from top to bottom is a reduced word for $w$.
Once again, there is an analog of the Coxeter-Knuth relations.  We
will need the details of this insertion map and relations for our main
theorems.  Our description of this map is based on an equivalent
algorithm in Tao Kai Lam's Ph.D. thesis \cite{lam1995b}.

First, there is an algorithm to insert a non-negative integer into a
unimodal sequence.  Given a number $k$ and a (potentially empty)
unimodal sequence $r = r_1 \dots r_l$ with middle index $j$, we insert
$k$ into $r$ and obtain another unimodal sequence as follows:

\begin{enumerate}

\item If $k \neq 0$ or $r_j \neq 0$, perform Edelman-Greene insertion
of $k$ into $r_{j+1} \dots r_l$. Call the bumped entry $k^-$, if it
exists.  Call the resulting string after insertion $v_{1}\dots v_{q}$. 
Note, $q$ may be $l-j$ or $l-j+1$.  

\item If $k = 0$ and $r_j =0$, set $k^- = 1$.  Set $v_{1}\dots v_{q}=r_{j+1} \dots r_l$.

\item If $k^-$ exists, perform Edelman-Greene insertion of $-k^-$ into
$-r_1 \dots -r_j$. This time a bumped entry $-r_{i}$ will exist, as $k^- > r_j$. Set 
$k'=r_{i}$.  Set $u_{1} \dots u_j$ to be the result of negating every entry
in the resulting string after insertion and reversing it. 

\item If $k^-$ does not exist, set $u_{1} \dots u_j=r_{1} \dots r_j$.

\item Output the unimodal sequence $u_{1} \dots u_jv_{1}\dots v_{q}$
  and $k'$ if it exists.  
\end{enumerate}

The \textit{Kra\'skiewicz insertion} of a non-negative integer $k$ into a shifted
unimodal tableau $P'$ starts by inserting $k$ into the first row of
$P'$ using the algorithm above.  Replace the first row of $P'$ by
$u_{1}\dots u_{p}v_{1}\dots v_{q}$.  If $k^{-}$ exists and $k'$ is the
output, then insert $k'$ in the second row of $P'$, etc.  Continue
until no output exists or no further rows of $P'$ exist.  In that
case, add $k'$ in a new final row along the diagonal so the result is
again a shifted unimodal tableau.  Call the final tableau
$P'\leftarrow k$.  For $w\in B_{\infty}$ and $a=a_{1}\dotsb a_{p}\in R(w)$,
let $P'(w)$ be the result of inserting $\emptyset \leftarrow a_{1}\dotsb
a_{p}$ consecutively into the empty shifted unimodal tableau denoted
$\emptyset$.

For example, using Kra\'skiewicz insertion, on the same reduced word $1343$ as before inserts to
give
\ytableausetup{aligntableaux=bottom, smalltableaux}
\[
\ytableaushort{1} \raisebox{.6ex}{ $\to$ }
\ytableaushort{13}  \raisebox{.6ex}{ $\to$ }
\ytableaushort{134}  \raisebox{.6ex}{ $\to$ }
\ytableaushort{\none  1,434}.
\]
So 
\[
\ytableausetup{aligntableaux=bottom, smalltableaux}
P'(1343)= \ytableaushort{\none  1,434} 
\hspace{.3in} \text{and} \hspace{.3in}
Q'(1343)=\ytableaushort{\none  4,123}.  
\]
\vspace{.3in}
Also, $0 2 1 0 3 2 1 0 1 \in R([\bar{3},\bar{4}, \bar{1}, 2])$ gives
$P'=U, Q'=T$ from Figure~\ref{fig: tableaux}.

Kra\'skiewicz insertion behaves well with respect to the peaks of a
reduced word.  Given any word $a_{1}\dots a_{p}$, we say $a$ has an
\textit{ascent} in position $0<i<p$ if $a_{i}<a_{i+1}$ and a
\textit{descent} if $a_{i}>a_{i+1}$.  Similarly, we say $a$ has a
\textit{peak} in position $1<i<p$ if $a_{i-1}<a_{i}>a_{i+1}$.  Define
the \textit{peak set} of $a \in R(w)$ to be $peaks(a)=\{1<i<p \given
a_{i-1}<a_{i}>a_{i+1} \}$.  For example, $peaks(4565)=\{3\}$ and
$peaks(7267)=\emptyset$.  Recall that standard tableaux have
associated ascent sets and descent sets as well as defined just before
Theorem~\ref{edelman1987balanced}.  Given a standard (shifted) tableau
$T$, we say $j$ is a peak of $T$ provided $j$ appears after $j-1$ and
$j+1$ in the row reading word of $T$, so there is an ascent from $j-1$
to $j$ and a descent from $j$ to $j+1$.  The peak set of $T$, denoted
again $peaks(T)$, is defined similarly.

\begin{thm}\cite[Theorem 2.10]{lam1995b}
\label{lem:same.peaks}
Given a signed permutation $w$ and a reduced word $a\in R(w)$, $peaks(a)=peaks(Q^\prime(a))$.
\end{thm}

One important tool for studying Kra\'skiewicz insertion is a
family of local transformations on words known as the \textit{type B
  Coxeter-Knuth moves}.  These moves are based on certain type B
elementary Coxeter relations that depend on exactly four adjacent
entries of a word.

\begin{defn}\label{d:ck.moves}\cite{kraskiewicz1989reduced} 
  The \textit{elementary Coxeter-Knuth moves of type B} are given by
  the following rules on any reduced word $i_{1}i_{2}i_{3}i_{4}$.  If
  $i_{1}i_{2}i_{3}i_{4}$ has no peak then $\beta
  (i_{1}i_{2}i_{3}i_{4})=i_{1}i_{2}i_{3}i_{4}$.  If
  $i_{1}i_{2}i_{3}i_{4}$ has a peak in position $3$, $\beta
  (i_{1}i_{2}i_{3}i_{4})$ is given by reversing $\beta
  (i_{4}i_{3}i_{2}i_{1})$.  If $i_{1}i_{2}i_{3}i_{4}$ has a peak in
  position 2, then we have three cases:

\begin{enumerate}
\item \textbf{Long braid}: If $i_{1}i_{2}i_{3}i_{4}=0101$, then define
$\beta (0101) = 1010$.  Note $1010$ is another reduced word for the
same signed permutation, and it has a peak in position 3.

\item \textbf{Short braid witnessed by smaller value}: If there are 3
  distinct letters among $i_{1}i_{2}i_{3}i_{4}$ and there is a
  corresponding short braid relation specifically of the form 
  $i_{1}i_{2}i_{3}i_{4} = a\ b+1 \ b \ b+1$ or $b \ b+1 \ b\ a$ for
  some $a<b$.  Define
\begin{align}\label{e:rule.2}
\beta (a\ b+1 \ b \ b+1) &=a\ b \ b+1 \ b  \text{  and  }\\
\beta (b \ b+1 \ b\ a) &= b+1 \ b \ b+1\ a. 
\end{align}

Again the sequence $\beta(i_{1}i_{2}i_{3}i_{4}) $ has a peak in
position 3 since $a<b$.  Also, this word is another reduced word for the same
signed permutation which differs by a short braid move.

\item \textbf{Peak moving commutation}: In all other cases, 
\[
\beta (i_{1}i_{2}i_{3}i_{4}) = (i_{1}i_{2}i_{3}i_{4}) s_{j}
\]
for the smallest $j$ such that $(i_{1}i_{2}i_{3}i_{4}) s_{j}$ is
related to $i_{1}i_{2}i_{3}i_{4}$ by a commuting move and has peak in
position 3.  Here $s_{j}$ is the operator acting on the right by
swapping positions $j$ and $j+1$.
\end{enumerate}
\end{defn}

Observe that $i_{1}i_{2}i_{3}i_{4}$ is fixed by $\beta$ if and only if
$i_{1}i_{2}i_{3}i_{4}$ has no peak.  Furthermore, the map $\beta$ is an involution on
$R(w)$ for $w$ a signed permutation with $\ell(w)=4$.  Define a family
of involutions $\beta_{i}$ 
acting on reduced words
$a_{1}a_{2}\dotsb a_{p}$ 
by replacing
$a_{i}a_{i+1}a_{i+2}a_{i+3}$ by $\beta(a_{i}a_{i+1}a_{i+2}a_{i+3})$,
 provided $0 < i \leq p-3$.

\begin{thm}\label{t:kras}\cite{kraskiewicz1989reduced} Let $a$ and $b$
  be reduced words of signed permutations.  Then $P'(a) = P'(b)$ if
  and only if there exist Coxeter-Knuth moves of type B relating $a$
  to $b$.  Furthermore, for each standard shifted tableau $Q$ of the
  same shape as $P'(a)$, there exists a reduced word $c$ for the same
  signed permutation such that $P'(c)=P'(a)$ and $Q'(c)=Q$.
\end{thm}

Using the Coxeter-Knuth moves of type B, we can define an
analogous graph $CK_B(w)$ on the reduced words for $w\in B_n$ with
edges defined by the involutions $\beta_{i}$.  Each connected
component of $CK_{B}(w)$ has vertex set given by a Coxeter-Knuth
equivalence class $\{a \in R(w) \given P'(a)=P' \}$, and assuming this
set is nonempty, $Q'$ gives a bijection between this set and the
standard shifted tableaux of the same shape as $P'$.  In
Section~\ref{s:Kraskiewicz}, we will show that every connected
component of $CK_{B}(w)$ is isomorphic to some $CK_B(v)$ where $v$ is
increasing.  In Section~\ref{s:axioms}, we will show that $Q'$ gives
an isomorphism of signed colored graphs with a graph on standard
shifted tableaux of the same shape with edges given by shifted dual
equivalence.

\subsection{Stanley symmetric functions revisited}\label{sub:transition}

For signed permutations, there are two forms of Stanley symmetric
functions and their related Schubert polynomials, see
\cite{BH,FK2,lam1995b}.  The distinct forms correspond to the root
systems of type B and C, which both have signed permutations as
their Weyl group.  The definition we will give is the type C
version, from which the type B version can be readily obtained.
First, we introduce an auxiliary family of quasisymmetric functions.

In type A, ascent sets of reduced words can be used to define the
Stanley symmetric functions.  In type B/C, the peak set of a reduced
word plays a similar role.

\begin{defn}\label{def:peak.quasi}\cite[Eq. (3.2)]{BH} Let $X =
  \{x_{1},x_{2},\dots \}$ be an alphabet of variables.  The
  \textit{peak fundamental quasisymmetric function} of degree $d$ on a
  possible peak set $P$ is defined by
\[
\Theta_{P}^{d} (X)= \sum_{(i_{1} \leq \dots \leq
i_{d}) \in A_{d}(P)} 2^{|\{i_{1},i_{2},\dots , i_{d} \}|}x_{i_{1}}
x_{i_{2}}\cdots x_{i_{d}}
\]
and $A_{d}(P)$ is the set of all \textit{admissible sequences} $(1\leq
i_{1} \leq \dots \leq i_{d})$ such that $i_{k-1}=i_{k}=i_{k+1}$ only
occurs if $k \not \in P$.
\end{defn}

The peak fundamental quasisymmetric functions also arise in
Stembridge's enumeration of $P$-partitions
\cite{stembridge1997enriched} and are a basis for the peak subalgebra
of the quasisymmetric functions as studied by
\cite{Billera.Hsiao.vanWilligenburg,schocker} and many others. They
are also related to the Schur $Q$-functions $Q_{\mu}(X)$ which are
specializations of Hall-Littlewood polynomials $Q_{\mu}(X;t)$ with
$t=-1$, see \cite[III]{Macdonald1995}.  By \cite[Prop. 3.2]{BH}, the
following is an equivalent definition of Schur $Q$-functions.

\begin{defn}\label{def:schur.Q} 
For a shifted shape $\mu$, the Schur $Q$-function $Q_{\mu }(X)$ is 
\[
Q_{\mu}(X ) = \sum_{T} \Theta_{peaks(T)}^{|\mu |}(X) 
\]
where the sum is over all standard shifted tableaux $T$ of shape
$\mu$.
\end{defn}

\begin{remark}
In this way, the peak fundamental quasisymmetric functions play the
role of the original fundamental quasisymmetric functions in Gessel's
expansion of Schur functions \cite{gessel1984}.
\end{remark}

Let $g_{w}^{\mu}$ 
be the number of distinct shifted tableaux of shape
$\mu$ that occur as $P'(a)$ for some $a \in R(w)$ under Kra\'skiewicz
insertion.  The numbers $g_{w}^{\mu}$ 
can equivalently be defined as
the number of reduced words in $R(w)$ mapping to any fixed standard
tableaux of shape $\mu$ by Haiman's promotion operator \cite[Prop. 6.1
and Thm. 6.3 ]{haiman1992dual}.  Haiman's promotion operator on $a \in
R(w)$ in type B is equivalent to Kra\'skiewicz's $Q'(a)$.  Recall
from Theorem~\ref{lem:same.peaks} that $Q'(a)$ and $a$ have the same
peak set which implies the equivalence in the following definition.

\begin{defn}\label{def:stan}\cite[Prop. 3.4]{BH}
For $w\in B_{\infty}$ with $d=\ell(w)$, define the type C \textit{Stanley
symmetric function} to be
\begin{align*}
F^C_{w}(X) &= \sum_{\mu} g_{w}^{\mu} 
Q_{\mu}(X)\\
& = \sum_{a \in R(w)} \Theta^{\ell(w)}_{peaks(a)}(X )\\
& = \sum_{a \in R(w)}\hspace{.1in} \sum_{(i_{1} \leq \dots \leq
i_{d}) \in A_{d}(P)} 2^{|\{i_{1},i_{2},\dots , i_{p} \}|}x_{i_{1}}
x_{i_{2}}\cdots x_{i_{d}}.
\end{align*}
\end{defn}

Every Schur $Q$-function is itself a type C Stanley symmetric
function.  In particular, for the shifted partition $\mu
=(\mu_{1}>\mu_{2}>\dots >\mu_{k}>0)$, we can construct an increasing
signed permutation $w(\mu)$ in one-line notation starting with the
negative values $\bar{\mu}_{1},\bar{\mu}_{2},\dots,\bar{\mu}_{k}$ and
ending with the positive integers in the complement of the set
$\{\mu_{1},\mu_{2},\dots,\mu_{k} \}$ in $[\mu_{1}]$.  For example, if $\mu=(5,3,1)$
then $w(\mu)=(\bar{5},\bar{3},\bar{1},2,4)$.  Then by
\cite[Thm.3]{BH},
\begin{equation}\label{eq:increasing}
F^C_{w(\mu )}(X)=Q_{\mu }(X).
\end{equation}
Conversely, every increasing signed permutation $w$ gives rise to an 
$F^C_{w}$ which is a single Schur $Q$-function defined by the negative
numbers in $[w(1),\dots ,w(n)]$.   


\begin{thm}\label{thm:transition.B}\cite[Cor. 9]{billey1998transition} Let $w$ be a signed permutation
which is not increasing.  Then we have the following transition
equation
\begin{equation}\label{eq:transition.B}
F^C_{w}(X) = 
\sum_{w' \in T(w)} F^C_{w'}(X).  
\end{equation}
This expansion terminates in a finite number of steps as a sum with all terms indexed by
increasing signed permutations.   
\end{thm}
Note that the index set $T(w)$ is defined in the remarks before Theorem~\ref{t:little.bijection}.

\begin{cor}\label{cor:reduced}
Let $w$ be a signed permutation that is not increasing.  Then
\[
|R(w) | =
\sum_{w' \in T(w)} |R(w')|.  
\]
\end{cor}

\begin{proof}
Consider the coefficient of $x_{1}x_{2}\cdots x_{\ell(w)}$ in
$F^C_{w}(x)$ and the right hand side of \eqref{eq:transition.B}.  
\end{proof}

\section{Pushes, Bumps, and the signed Little Bijection}\label{s:signed.little}

In this section, we define the signed Little map on reduced words via
two other algorithms called \emph{push} and \emph{bump}.  A key tool is the
wiring diagrams for reduced words of signed permutations.  The main
theorem proved in this section is Theorem~\ref{t:little.bijection},
which says that the Little bumps determine 
a bijection on reduced
words 
that realizes the transition equation for type C Stanley
symmetric functions.

\begin{figure}
    \caption{A wiring diagram of $a=0120312 \in R(34\bar 2 \bar 1)$. 
\label{fig:wiring diagram}}
\begin{center}
\includegraphics[scale=0.8]{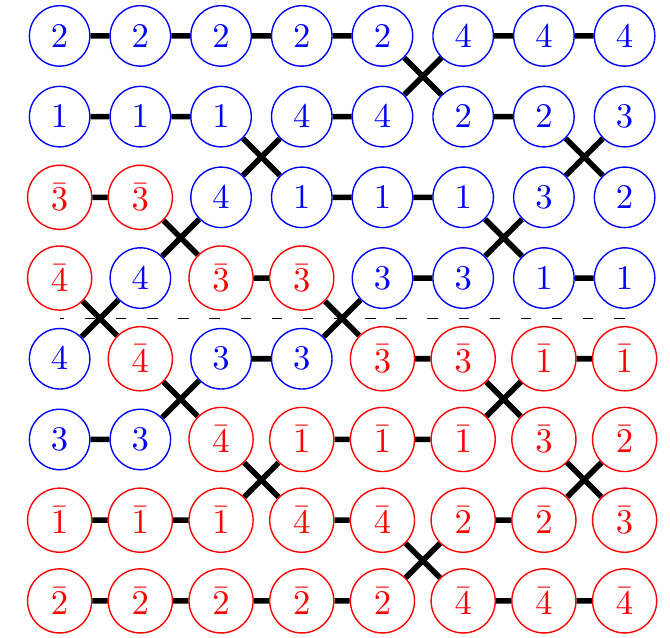}
\end{center}
\end{figure}

The \emph{wiring diagram} of $a=a_1a_2\cdots a_p$ is the array $[p]
\times [-n,n]$ in Cartesian coordinates.  Each ordered pair in the
array indexes a square cell, which may or may not contain a cross,
denoted $\times$: specifically, the crosses are located at $(j,a_j)$,
and $(j,-a_j)$ for all $1\leq j\leq p$ (thus if
$a_j=0$, there will be just one cross in column $j$).  The boundaries
between cells $(u,v)$ and $(u, v+1)$, as well as the top and bottom
edges of the diagram, contain a horizontal line denoted
$\underline{\hspace{.1in}}$ unless there is a cross in the cell
immediately above or below.  The line segments connect with the
crosses to form ``wires'' with labels $[-n] = \{-1, -2, \dots, -n\}$
and $[n] = \{1, 2, \dots, n\}$ starting on the right hand edge of the
diagram.  So, specifically in the rightmost column, if $a_p>0$ then
wires $a_p$ and $a_{p}+1$ cross in column $p$, and also wires $-a_p$
and $-a_{p}-1$ cross.  If $a_p=0$, then wires $1$ and $-1$ cross in
column $p$.  For each $0\leq k<p$, define $w^k = s_{a_p} \dots
s_{a_{k+1}}$, so that $w^0=w^{-1}$ and define $w^p$ to be the identity. The
sequence of wire labels reading up along the left edge from bottom to
top gives the long form of the signed permutation $w^{-1}$.  More
generally, the sequence of labels on the wires of the wiring diagram
just to the right of column $k$ is the signed permutation $w^k$.
Every wiring diagram should be considered as a subdiagram of the
diagram with wires labeled by all of $\mathbb{Z} \setminus \{0 \}$
where all constant trajectories above and below the diagram are suppressed in keeping with the
notion that every signed permutation in $B_n$ can be thought of as an
element of $B_\infty$.  See Figure~\ref{fig:wiring diagram} for an
illustration of these definitions.

The \emph{inversion set} $\Inv_B(w)$ of a signed permutation $w$ is
\[
\Inv_B(w) = \{ (i,j) \in ([-n] \cup [n]) \times [n] : \ |i| \leq j \ \mbox{and} \ w(i) > w(j)\}.
\]
We have defined the wiring diagrams so that the 
inversion $(i,j)$ corresponds with the crossing of wires $i$ and
$j$ in any wiring diagram of a reduced word for $w$.  Note, the wires
$-j$ and $-i$ also cross in the same column in such a diagram.  Thus,
it is equivalent to refer to the inversion $(i,j)$ by
$(-j,-i)$.  If $a \in R(w)$, then the wiring diagram for $a$ is
\textit{reduced} and every crossing corresponds 
to an inversion for
$w$.

For a word $a = a_1 \dots a_p$ and $\delta \in \{-1,1\}$, we define a
\emph{push $P^\delta_i$} at index $i$ to be the map that adds $\delta$
to $a_i$ while fixing the rest of the word provided $a_i$.  If
$a_{i}=0$, then regardless of $\delta$, the $i$th entry is set to $1$
in the resulting word e.g. $P^{-}_{1}(0)=P^{+}_{1}(0)=1$. We will write $P^{-}_{k}$ and $P^{+}_{k}$ for
$\delta = -1$ and $\delta = 1$ respectively.  The effect pushes have
on wiring diagrams can be observed in Figure~\ref{fig:fullbump}.

If $a = a_1 \dots a_p$ is a word that is not reduced, we say a
\emph{defect} is caused by $a_i$ and $a_j$ with $i \neq j$ if the
removal of either leaves a reduced word.   The following lemma can be deduced for signed permutations from the wiring diagrams, but it holds more generally for  Coxeter groups.

\begin{lem}\cite[Lemma 21]{lam2006little} \label{lem:lam} For $W$ a Coxeter group and $w\in W$, let $a =
a_1 \dots \hat{a_i} \dots a_p \in R(w)$ such that $a_1 \dots a_p$ is
not reduced.  Then there exists a unique $j \neq i$ such that $a_1
\dots \hat{a_j} \dots a_p$ is reduced.  Moreover, $a_1 \dots \hat{a_j}
\dots a_p\in R(w)$.

\end{lem}

\begin{defn}[\textbf{Little Bump Algorithm}]
Let $a = a_1 \dots a_p$ be a reduced word of the signed permutation
$w$ and $(i,j) \in \Inv_B(w)$ such that $\ell(w t_{ij}) = \ell(w)-1$.  
Fix $\delta \in \{-1,1 \}$.  We define the \emph{Little bump for $w$}
at the inversion $(i,j)$ in the direction $\delta$, denoted
$B^{\delta}_{(i,j)}$, as follows. 
\begin{enumerate}

\item[Step 1:] Identify the column $k$ and row $r$ containing the wire
  crossing $(i,j)$ with $i<j$.  If $a_k=0$, set $b:= P^{1}_k(a)$ and
  $\delta:=-\delta$.  If $a_k>0$, then either $w^{k}(a_k) = i$ or
  $w^{k}(a_k) = -j$. If $w^{k}(a_k) = i$, set $b:= P^\delta_k(a)$.
  Otherwise $w^{k}(a_k) = -j$, and we set $\delta := -\delta$ and
  $b := P^{\delta}_k(a)$.  Next, set $r:=r+\delta$. Note that
  the order in which the variables are updated matters.  Let $(x<y)$
  be the new wires crossing in column $k$ and row $r$.

\item[Step 2:] If $b$ is reduced, return $b$. Otherwise, by
  Lemma~\ref{lem:lam} there is a unique defect caused by $b_k$ and
  some $b_l$ with $l \neq k$. If $b_l>0$, then either $w^{l}(b_l+1)\in
  \{x,y\}$ or $w^{l}(b_l+1)\in \{-x,-y\}$.  If $b_l=0$, then $x=-y$
  and $w^{l}(1)\in \{-x,x\}$.

\begin{itemize}
\item If $b_l>0$ and $w^{l}(b_l+1)\in \{x,y\}$, set $r:=r+\delta$,
  $k:=l$, and $b := P^\delta_k(b)$.  After updating the variables, let
  $(x<y)$ be the wires crossing in the diagram for $b$ in column
  $k$ and row $r$.  Repeat Step 2.

\item Otherwise, $b_{l}=0$ or $w^{l}(b_l+1)\in \{-x,-y\}$.  Set $\delta :=
  -\delta$, $r:=r+\delta$, $k := l$ and $b := P^{\delta}_k(b)$.  Again,
  the order matters.  After updating the variables, let $(x<y)$ be
  the wires crossing in column $k$ and row $r$.  Repeat Step 2.
\end{itemize}
\end{enumerate}
\end{defn}

Figure~\ref{fig:fullbump} shows each step of a Little bump in terms of
wiring diagrams.  The corresponding effect on reduced words can be
read off the diagrams by noting the row numbers of the wire crossings
in the upper half plane including the $x$-axis.   

\begin{remark}\label{rem:bump.1}
  The Little bump algorithm is best thought of as acting on wiring diagrams.
  At every step, 
  the pushes move
  the $(x,y)$-crossings consistently in the 
  initial direction of $\delta$. In the
  first step, we move the $(i,j)$-crossing in the wiring diagram up if
  $\delta =+1$ and down if $\delta = -1$.  If wires $i$ and $j$ cross
  in the upper half plane then the swap $a_k$ is replaced with
  $a_k+\delta$.  However, if wires $i,j$ cross in the lower half plane
  then the swap $a_k$ is replaced with $a_k-\delta$ and the sign of
  $\delta$ is switched.  If a new defect crossing is later found on
  the other side of the $x$-axis from the last crossing, then the sign
  of $\delta$ will switch again so that the crossing continues to move in the same
  direction.  Thus, if the initial push moved $(i,j)$ down, each
  subsequent iteration will continue to move a crossing down, but the
  effect on the word from the corresponding push can vary.
\end{remark}

\begin{remark}\label{rem:bump.2}
  Observe that in each iteration of Step 2, the word $b$ has the
  property that its subword $b_1 b_2 \cdots \widehat{b_k} \cdots b_p$
  is reduced.    
\end{remark}

When analyzing Little bumps and pushes, we will need to track where
the next defect can occur.  Given the wiring diagram for a word $b=b_1
\cdots b_p$, not necessarily reduced, and a crossing $(x,y)$ in column
$k$ in the diagram, define the \textit{(lower) boundary} of $b$ for
the crossing $(x,y)$, denoted $\partial^k_{(x,y)}(b)$, to be the union
of the trajectory of $y$ from columns 0 to $k$ and the trajectory of
$x$ from columns $k$ to $p$.  Note using the notation of Step 2 above,
if a defect is caused by $b_{k}$ in this iteration it will occur along
$\partial^k_{(x,y)}(b)$.  In Figure~\ref{fig:fullbump}, the boundary
of each crossing that will be pushed is dashed.  A similar concept of
an upper boundary could be defined if the initial step pushes the
$(i,j)$ cross up.
that

\begin{lem}
  \label{lem:terminate} Let $a\in R(w)$ and $B^\delta_{(i,j)}$ be a
  Little bump for $w$ consisting of the sequence of pushes
  $P^{\delta_1}_{t_1}, P^{\delta_2}_{t_2}, \dots$ acting on $a$.  Then,
  for all $k$ and $\delta \in \{-1,1\}$, the push $P^{\delta}_k$
  appears at most once in this sequence.  Hence, the Little bump
  algorithm terminates in at most $2 \ell(w)$ pushes.
\end{lem}

This proof is a slight extension of the proof of Lemma 5
in~\cite{little2003combinatorial}.
\begin{proof}
  Let $a= a_1 a_2 \dots a_p \in R(w)$ and $a_k$ denote the swap
  introducing the inversion $(i,j)$, with $i<j$.  Since $B^-_{(i,j)} =
  B^+_{(-j,-i)}$, we need only demonstrate the result when the
  algorithm starts with a push that moves the $(i,j)$-crossing down.

  In Step 1 of the Little bump algorithm, either $b=P^\delta_k(a)$ is
  reduced or there is some $l \neq k$ such that $b_k $ and $b_l$ cause
  a defect.  Suppose the latter case holds.  Then $b_{k}$ and $b_{l}$
  swap the wire $i$ with some wire $h \neq i,j$.  By considering the
  reverse of the word if necessary, we may assume $k<l$.  The defect
  in column $l$ must occur on the boundary of $\partial^k_{(x,y)}(b)$.
  Observe that $\partial^k_{(i,j)}(a)$ and $\partial^l_{(i,h)}(b)$
  coincide from 1 to $k$ and from $l$ to $p$.  Moreover, between $k$
  and $l$, the boundary $\partial^l_{(i,h)}(b)$ is strictly lower in
  the wiring diagram than $\partial^k_{(i,j)}(a)$.  This can be seen
  in the first two diagrams shown in Figure~\ref{fig:fullbump}.
  Observe that the trajectories of $-i$ and $-j$ will not interact
  with $\partial^l_{(i,h)}(b)$ unless $j=-i$. Therefore the boundary
  $\partial^l_{(i,h)}(b)$ is weakly below the boundary
  $\partial^k_{(i,j)}(a)$.  Similar reasoning shows that on each
  iteration of Step 2 in the algorithm, the boundary always moves
  weakly down provided the initial push moves a crossing down.

\begin{figure}
\caption{The sequence of pushes corresponding to
the Little bump $B^-_{(\overline{2},1)}$ as applied to $a = 1021201 \in
R([1,\overline{3},\overline{2}])$. The boundary of each
crossing about to be moved is dashed. Here, red is negative and blue
is positive.  The thin dashed line
through the center row is row 0 and the row numbers increase going up.  
\label{fig:fullbump} 
}  \centering \vspace{10pt}
\begin{tabular}{cc}
$a=1021201$ & $P^-_3(a)=1011201$ \\

\includegraphics[page = 1, scale = 1.1]{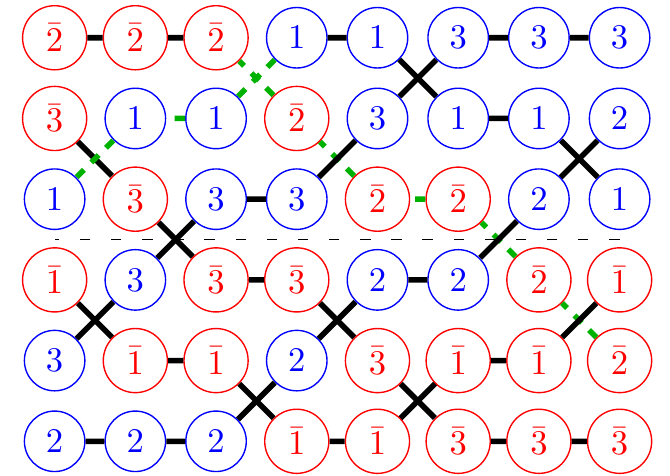}
&
\includegraphics[page = 2, scale = 1.1]{TypeBbump.pdf} \\

\\
$P^-_4P^-_3(a)=1010201$ & $P^+_6P^-_4P^-_3(a)=1010211$ \\

\includegraphics[page = 3, scale = 1.1]{TypeBbump.pdf}
&
\includegraphics[page = 4, scale = 1.1]{TypeBbump.pdf} \\
\\
$P^+_7P^+_6P^-_4P^-_3(a)=1010212$ & $P^-_1P^+_7P^+_6P^-_4P^-_3(a)=0010212$ \\

\includegraphics[page = 5, scale = 1.1]{TypeBbump.pdf}
&
\includegraphics[page = 6, scale = 1.1]{TypeBbump.pdf} \\

\end{tabular}

\end{figure}

\begin{figure}

\centering
\vspace{10pt}

\begin{tabular}{cc}

$P^+_2P^-_1P^+_7P^+_6P^-_4P^-_3(a)$ & $P^+_3P^+_2P^-_1P^+_7P^+_6P^-_4P^-_3(a)$\\

\includegraphics[page = 7, scale = 1.1]{TypeBbump.pdf}
&
\includegraphics[page = 8, scale = 1.1]{TypeBbump.pdf} \\

\end{tabular}
\vspace{10pt}

$B^-_{(\overline{2},1)} (a) =
P^+_5P^+_3P^+_2P^-_1P^+_7P^+_6P^-_4P^-_3(a)=0120312 \in
R([3,4,\overline{2},\overline{1}])$

\includegraphics{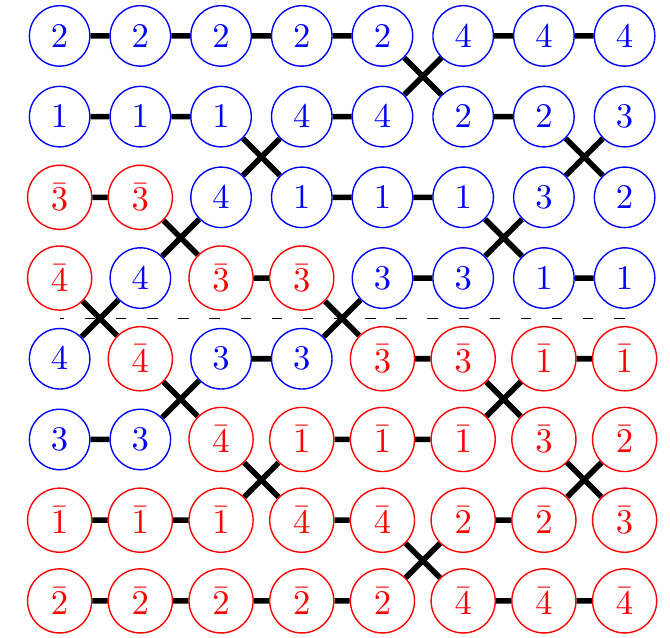}

\end{figure}

Now we can verify that no push occurs twice in the Little bump
algorithm.  In particular, we claim that both $P_k^-$ and $P_k^+$ can
occur, but they are never repeated.  For example, $P^-_5$ and $P^+_5$
both occur in Figure~\ref{fig:fullbump}.
To do this, we need to examine the argument above more closely.  Assume that Step 2 starts
with a push $P_k^\delta$ moving a crossing into position 
$(k,r)$ in 
the wiring diagram of $b$.  Assume this wiring diagram has a defect in
columns $k$ and $l$ and $k<l$.  Then the boundary before and after the
push $P_k^\delta$ agree weakly to the left of column $k$.  If
successive pushes occur strictly to the right of column $k$, then none
of these pushes will repeat $P_k^\delta$.  Furthermore, the boundary
to the left of column $k$ will be constant.  The first time another
iteration of Step 2 finds a defect weakly to the left of column $k$,
we claim it must occur in a column strictly to the left of column $k$,
thus the boundary moves strictly below 
$(k,r)$.  The reason column $k$
cannot be part of the defect this time is that the boundary has
negative slope just to the right of the crossing in row $r$ column
$k$, but to create a defect with a string passing below 
$(k,r)$ the
boundary must have positive slope where the two strings meet to the
right of column $k$.  Furthermore, if another push in column $k$
occurs later in the algorithm, it must be on the other side of the
$x$-axis so it must be $P_k^{-\delta}$ since the boundary moves
monotonically.  Once the boundary has moved beyond both crossings in
column $k$, neither crossing will be pushed again, so $P_k^\delta$
occurs at most once in the Little Bump algorithm.
\end{proof}

\begin{lem}\label{lem:peaks}
Let $w$ be a signed permutation, $a = a_1 \dots a_p \in R(w)$,  and $B^\delta_{(i,j)}$ be a Little bump for $w$.
Then 
\[
peaks(a) = peaks(B^\delta_{(i,j)}(a)).
\]
\end{lem}

\begin{proof}
  Say $i \in peaks(a)$, so $a_{i-1} < a_i > a_{i+1}$.  The statement
  holds unless a push applied to one of these entries during the
  algorithm leaves an equality in the resulting word $b$, say $b_{i-1}
  = b_i$.  In this case, there is a defect caused by $b_{i-1}$ and
  $b_i$, so we push the other next.  The direction of the new push for
  defects caused by adjacent entries will be the same unless $b_{i-1}
  = b_i=0$.  This cannot occur since $a_i > a_{i-1}\geq 0$ and if $a_i
  =1$, then $a_{i-1} = 0 = a_{i+1}$. Hence, $a_1 \dots \hat{a_i} \dots a_p$
  would not be reduced, which is not possible by
  Remark~\ref{rem:bump.2}.
\end{proof}

\begin{thm}[Restatement of Theorem~\ref{t:Q.preserving}, part 1] \label{thm:bump.image}
  Let $x \in B_\infty$, and let $a \in R(x)$. Say $j<k$.
\begin{enumerate}
\item If $B_{(j,k)}^-$  is a Little bump for $x$, then
  $B_{(j,k)}^-(a) \in R(xt_{jk}t_{ij})$ for some $i<j$.  
\item If $B_{(j,k)}^+$  is a Little bump for $x$, then
  $B_{(j,k)}^+(a) \in R(xt_{jk}t_{kl})$ for some $k<l$.  
\end{enumerate}
\end{thm}

\begin{proof}
When applying $B^{-}_{(j,k)}$ on $a$, the initial push is some
$P^{\delta_1}_{l_1}$ with $a_{1}\cdots \widehat{a_{l_{1}}} \cdots
a_{p} \in R(v)$ for some $v \in B_\infty$.  Let $b=P^{\delta_1}_{l_1}(a)$.  Then, one can
observe from the wiring diagrams that $s_{b_1}\cdots s_{b_{p}}= v
t_{ji_1}$ for some $i_1 \neq j,k$.

If $b$ is reduced, the bump is done.  Otherwise, by the Little bump
algorithm, there is some unique defect between $b_{l_1}$ and $b_{l_2}$
so we push next in column $l_{2}$.  We know $b_{1}\cdots
\widehat{b_{l_{2}}} \cdots b_{p} \in R(v)$ by construction and
Remark~\ref{rem:bump.2}. So when the next push occurs in column $l_2$
the new crossing will be between $j$ and another string $i_2$.
Continuing the algorithm, we see recursively that $B^{-}_{(j,k)}(a)
\in R(vt_{ji})$ for some $i\neq j,k$.

Assume for the sake of contradiction that $i>j$, and say the
$(i,j)$-crossing in the wiring diagram of $c=B^{-}_{(j,k)}(a)$ occurs
in column $l$.  By removing the $l^{th}$ swap from $c$ we get a wiring
diagram for $v$ that does not have $(j,i)$ as an inversion.  Thus, the
$i$-wire must stay entirely above the $j$-wire. Hence, the $i$ wire is
above the boundary of the last push. Thus, it cannot be a part of the
last push since the boundary moves monotonically according to the
proof of Lemma~\ref{lem:terminate}.  We can then conclude that $c \in
R(v t_{ij})$ for some $i<j$.

A similar proof holds for the second statement.
\end{proof}

We recall the notation of transition equations from
Section~\ref{s:intro}. If $w$ is not increasing, let $r$ be the
largest value such that $w_r>w_{r+1}$.  Define $s$ so that $(r<s)$ is
the lexicographically largest pair of positive integers such that
$w_{r}>w_{s}$.  Set $v=wt_{rs}$.  Let $T(w)$ be the set of all signed
permutations $w'=vt_{ir}$ for $i<r$, $i\neq 0$ such that
$\ell(w')=\ell(w)$.

\newcommand{\upset}{U}
\newcommand{\downset}{D}
Next, we show that the \textit{canonical} Little bump $B_{(r,s)}^{-}$ for $w$ respects the transition
equations in Theorem~\ref{thm:transition.B}.  This is best done by
describing the domain and range of Little bumps in greater generality.
For $v \in B_\infty $ and $j \in \mathbb{Z}-\{0 \}$, we define
\begin{align*}
\downset(v,j) & = \{vt_{ij} : \ i<j,\ i\neq 0 \ \text{and} \ \ell(vt_{ij}) = \ell(v)+1\} \\
\upset(v,j) & = \{vt_{jk} : \ j<k,\ k\neq 0 \ \text{and} \ \ell(vt_{jk}) = \ell(v)+1\}.
\end{align*}
Observe that we have $\downset(v,-j) = \upset(v,j)$.
We now prove the analog of \cite[Theorem 3]{little2003combinatorial}, from which we can deduce Theorem~\ref{t:little.bijection}.

\begin{lem}
\label{lem:bijective.transition}
Let $v \in B_\infty $ and $j \neq 0$.
Then 
\[
\sum_{x \in \upset(v,j)} |R(x)| =\sum_{y \in \downset(v,j)} |R(y)|.
\]
\end{lem}

\begin{proof} We will prove the equality bijectively by using a
  collection of Little bumps.  Define a map $M_{v,j}$ on $\cup_{x \in
    \upset(v,j)} R(x)$ as follows.  Say $a = a_1 \dots a_p \in R(x)$
  for some $x \in \upset(v,j)$.  Then $x=vt_{jk}$ for some unique
  $k>j$.  Furthermore, $B^{-}_{(j,k)}$ is a Little bump for $x$.  By
  Theorem~\ref{thm:bump.image}, we know that $B^{-}_{(j,k)}(a) \in
  R(vt_{ij})$ for some $i<j$ and $\ell(vt_{ij}) = \ell(v
  t_{jk})$. Thus, $v t_{ij} \in \downset(v,j)$.  Set $M_{v,j}(a) :=
  B^{-}_{(j,k)}(a)$ for all $a \in R(x)$.  In this way, we construct a
  map
$$
M_{v,j}: \cup_{x \in \upset(v,j)} R(x) \longrightarrow
\cup_{y \in \downset(y,j)} R(y).  
$$ 
Since the Little bump algorithm is reversible with
$B^+_{(i,j)}(c)=a$ in the notation above, we know $M_{v,j}$ is injective.

The bijective proof is completed by observing that
 $\downset(v,j) = \upset(v,-j)$, $\upset(v,j) =
 \downset(v,-j)$, and that
$B^-_{(-j,-i)} = B^+_{(i,j)}$ is a Little bump for $v t_{ij}
\in \downset(v,j)$ whose image, by the above argument, is a
reduced word of some $x \in \upset(v,j)$.  
\end{proof}

\begin{cor}[Restatement of Theorem~\ref{t:little.bijection}]
\label{c:little.bijection}
Let $a \in R(w)$ and $B^{-}_{(r,s)}$ be the canonical Little bump for $w$. 
Recall that $(r,s)$ is the  lexicographically last inversion in $w$.
Then $B^{-}_{(r,s)}(a)$ is a reduced word for $w'$ where
\[
w' \in T(w)=\{ w t_{rs} t_{lr}\ \mid\ l<r\ \mbox{and}\ \ell(w) = \ell(w t_{rs} t_{rl})\}.
\]
\end{cor}

\begin{proof}
Observe $\upset(w t_{rs},r) = \{w\}$ and 
\[
\downset(w t_{rs},r) = \{ w t_{rs} t_{lr}\ \mid\ l<r,\ l \neq 0, \ \mbox{and}\ \ell(w) = \ell(w t_{rs} t_{rl})\}=T(w).
\]
The result now follows from Lemma~\ref{lem:bijective.transition} with
$v=wt_{rs}$ since $\ell(wt_{rs})=\ell(w)-1$ by choice of $(r,s)$.  
\end{proof}

\section{Kra\'{s}kiewicz insertion and the signed Little Bijection}\label{s:Kraskiewicz}

In this section, we show that Coxeter-Knuth moves act on $Q'(a)$ by
shifted dual equivalence, as defined in~\cite{haiman1992dual}.  We
then prove the remainder of Theorem~\ref{t:Q.preserving} by applying
properties of shifted dual equivalence and showing that Little bumps
and Coxeter-Knuth moves commute on reduced words of signed
permutations.

\newcommand{\fl}{\mathrm{fl}}

For a permutation $\pi \in S_{n}$, let $\pi|_I$ be the subword
consisting of values in the interval $I$.  Let $\fl(\pi|_I) \in
S_{|I|}$ be the permutation with the same relative order as
$\pi|_{I}$. Here $\mathrm{fl}$ is the \textit{flattening operator}.  Similarly,
for $Q$ a standard shifted tableau $Q|_I$ denotes the shifted skew
tableau obtained by restricting the tableau to the cells with values
in the interval $I$.  

\begin{defn}\cite{haiman1992dual}
Given a permutation $\pi\in S_n$, define the \emph{elementary
shifted dual equivalence} $h_i$ for all $1\leq i\leq n-3$ as follows.
If $n\leq 3$, then $h_1(\pi)=\pi $.
If $n=4$, then $h_1(\pi)$ acts by swapping $x$ and $y$ in the cases below,
\begin{equation}
1x2y \quad x12y \quad 1x4y \quad x14y \quad 4x1y \quad x41y \quad 4x3y \quad x43y,
\end{equation}
and $h_1(\pi) = \pi$ otherwise. 
If $n>4$, then $h_i$ is the involution that fixes values not in $I=\{i, i+1, i+2, i+3\}$ and permutes the values in $I$ via $\fl(h_i(\pi)|_I)= h_1(\fl(\pi|_I))$.
\end{defn}
\noindent
As an example, $h_1(24531)=14532$, $h_2(25134)=24135$, and $h_3(314526)=314526$.

Recall from Definition~\ref{d:ck.moves} that a
type B Coxeter-Knuth move starting at position $i$ is denoted by
$\beta_i$.
One can verify that this definition is equivalent to defining $h_i$ as
\[
h_i(\pi)=(\beta_i(\pi^{-1}))^{-1}.
\]

Given a standard shifted tableau $T$, we define $h_i(T)$ as the result
of letting $h_i$ act on the row reading word of $T$.  Observe $h_i(T)$ is
also a standard shifted tableau. We can define an
equivalence relation on standard shifted tableaux by saying $T$ and
$h_{i}(T)$ are \textit{shifted dual equivalent} for all $i$.

\begin{thm}\cite[Prop.~2.4]{haiman1992dual}\label{thm: Haiman}
Two standard shifted tableaux  are shifted dual equivalent if and only if they have the same shape.
\end{thm}

Recall the notion of
jeu de taquin is an algorithm for sliding one cell at a time in a
standard tableau on a skew shape in such a way that the result is
still a standard tableau \cite{sagan1991}.  The analogous notion
for shifted tableaux was introduced independently in~\cite{sagan1987shifted}
and~\cite{worley1984theory}.

\begin{lem}\cite[Lemma 2.3]{haiman1992dual}\label{lem: slides and h commute}
Given two standard shifted tableaux $T$ and $U$ with $T=h_i(U)$,
let $T^\prime$ and $U^\prime$ be the result of applying any fixed
sequence of jeu de taquin slides to $T$ and $U$, respectively.  Then
$T^\prime=h_i(U^\prime)$.
\end{lem}

\begin{defn}\label{def:delta}
Given a standard shifted tableau $Q'$, define $\Delta(Q')$ as the
result of removing the cell containing 1, performing jeu de taquin
into this now empty cell, and subtracting 1 from the value of each of
the cells in the resulting tableau.
\end{defn}

\begin{lem}\cite[Theorem 3.24]{lam1995b}\label{lem: delta on Q}
Let $w$ be a signed permutation and $a = a_1 \cdots a_p \in R(w)$. 
Then under Kra\'{s}kiewicz insertion
\begin{equation}
 Q'(a_2 \cdots a_p) = \Delta (Q'(a_1 \cdots a_p)).
\end{equation}
\end{lem}

\begin{lem}
\label{prop: beta on Q} 
Let $w$ be a signed permutation, and let $a = a_1 \dots a_p \in R(w)$. 
Then $Q'(\beta_i(a))=h_i(Q'(a))$ for all integers $1\leq i \leq p-3$. 
\end{lem}

\begin{proof}
  Recall that $\beta_i$ acts trivially on $a$ if and only if both
  $i+1, i+2 \notin peaks(a)$.  Similarly, $h_i$ acts trivially on
  $Q'(a)$ if and only if both $i+1, i+2 \notin peaks(Q'(a))$.  By
  Theorem~\ref{lem:same.peaks}, we then see $\beta_i$ acts trivially
  if and only if $h_i$ acts trivially.  Thus, the lemma holds if both
  $h_i$ and $\beta_i$ act trivially so we will assume that this is not the case.  

Since type B Coxeter-Knuth moves preserve Kra\'{s}kiewicz insertion
tableaux, we see $Q'(a)|_{[1,i-1]}=Q'(\beta_i(a))|_{[1,i-1]}$,
$Q'(a)|_{[i+4,p]}=Q'(\beta_i(a))|_{[i+4,p]}$ and that the shape of
$Q'(a)|_{[i,i+3]}$ and $Q'(\beta_i(a))|_{[i,i+3]}$ are the same.  In
particular, $Q'(a)$ differs from $Q'(\beta_i(a))$ by some
rearrangement of the values in $[i, i+3]$.  We need to show that this
rearrangement is the elementary shifted dual equivalence $h_i$.
The following proof of this fact is presented as a commuting diagram in Figure~\ref{fig: Q diagram}.

By omitting any extra values at the end of $a$, we may assume 
that $p=i+3$.  Now consider the tableaux $T$ and $U$
obtained by adding $i-1$ to each entry in $\Delta^{i-1}(Q'(a))$ and
$\Delta^{i-1}(Q'(\beta_i(a))$.  Because
$Q'(a)|_{[1,i-1]}=Q'(\beta_i(a))|_{[1,i-1]}$, it follows from the
definition of $\Delta$ that there is some fixed set of jeu de taquin
slides that relates both $Q'(a)|_{[i,i+3]}$ to $T$ and
$Q'(\beta_i(a))|_{[i,i+3]}$ to $U$.  Applying Lemma~\ref{lem: slides and
h commute}, we need only show that $T=h_i(U)$ to complete the proof.

By Lemma~\ref{lem: delta on Q}, we see
\[
\Delta^{i-1}(Q'(a))=Q'(a_ia_{i+1}a_{i+2}a_{i+3})
\]
 and 
 \[
 \Delta^{i-1}(Q'(\beta_i(a))) = Q'(\beta_1(a_ia_{i+1}a_{i+2}a_{i+3})).
 \] 

 Since $Q'(\beta_{1}(a_ia_{i+1}a_{i+2}a_{i+3}))$ and
 $Q'(a_ia_{i+1}a_{i+2}a_{i+3})$ are distinct by assumption and are
 necessarily standard tableaux of the same shifted shape with four
 cells, the shape must be $(3, 1)$.  Furthermore, there are only two
 standard tableaux of shifted shape $(3,1)$, so the two tableaux must
 be related by $Q'(\beta_{1}(a_{1}a_{2}a_{3}a_{4})) = h_{1}
 (Q'(a_{1}a_{2}a_{3}a_{4}))$.  Adding $i-1$ to each entry of the two
 tableaux in this equation changes $h_1$ to $h_i$ and yields the
 desired result, $T=h_i(U)$.
 %
%
\end{proof}

\begin{figure}[h]
\[
\xymatrixcolsep{2.3pc}\xymatrixrowsep{4pc}\xymatrix{
&Q'(a)|_{[i,i+3]}\ar@{<.>}[r]^-{h_i} \ar@{<->}[d]^-{j.d.t.}  \ar@{->}[dl]_-{\Delta^{i-1}}&Q'(\beta_i(a))|_{[i,i+3]} \ar@{<->}[d]^-{j.d.t.}\ar@{->}[dr]^-{\Delta^{i-1}}&\\
Q'(a_i\cdots a_{i+3}) \ar@{->}[r]^-{+(i-1)} \ar@{<->}@/_1.8pc/[rrr]^-{h_1}& T \ar@{<->}[r]^-{h_i} & U \ar@{<-}[r]^-{+(i-1)} & Q'(\beta_1(a_i\cdots a_{i+3}))\\
}
\] \vspace{.2in}\[
\xymatrixcolsep{2.5pc}\xymatrixrowsep{4pc}\xymatrix{ \ytableausetup{boxsize=normal, aligntableaux=center}
&\ytableaushort{{\scriptstyle i}{\scriptstyle i+1}{\scriptstyle i+3},\none\none{\scriptstyle i+2}}\ar@{<.>}[r]^-{h_i} \ar@{<->}[d]^-{j.d.t.}  \ar@{->}[dl]_-{\Delta^{i-1}}&\ytableaushort{{\scriptstyle i}{\scriptstyle i+2}{\scriptstyle i+3},\none\none{\scriptstyle i+1}} \ar@{<->}[d]^-{j.d.t.}\ar@{->}[dr]^-{\Delta^{i-1}}&\\
\ytableausetup{aligntableaux=bottom}
\ytableaushort{\none4,123} \ar@{->}[r]^-{+(i-1)} \ar@{<->}@/_1.8pc/[rrr]^-{h_1}& \ytableaushort{\none{\scriptstyle i+3},{\scriptstyle i}{\scriptstyle i+1}{\scriptstyle i+2}} \ar@{<->}[r]^-{h_i} & \ytableaushort{\none{\scriptstyle i+2},{\scriptstyle i}{\scriptstyle i+1}{\scriptstyle i+3}} \ar@{<-}[r]^-{+(i-1)} & \ytableaushort{\none3,124}\\
}
\]
\caption{The commuting relationships in the proof of Proposition~\ref{prop: beta on Q} on top and a generic example on bottom.
\label{fig: Q diagram}
}
\end{figure}
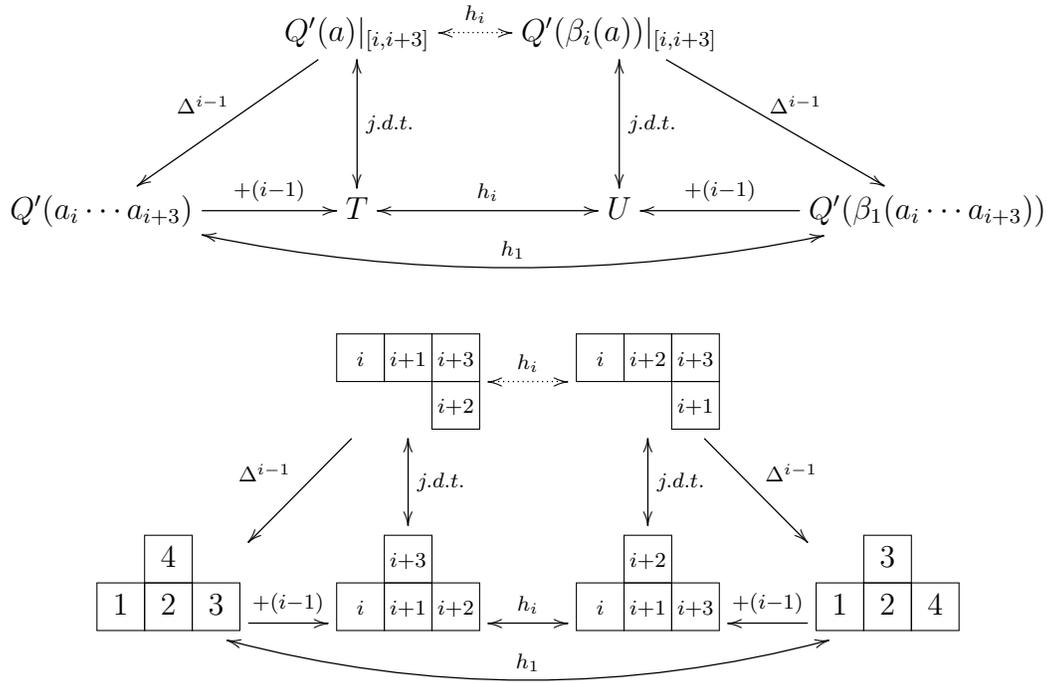

Next, we show that Coxeter-Knuth moves commute with Little bumps.

\begin{lem}
\label{lem:ck-little.commute}

Let $a = a_1 \dots a_p$ be a reduced word of the signed permutation $w$, $\beta_k$ a Coxeter-Knuth move for $a$ and $B^\delta_{(i,j)}$ a Little bump for $w$.
Then
\[
B^\delta_{(i,j)}(\beta_k(a)) = \beta_k(B^\delta_{(i,j)}(a)).
\]

\end{lem}

\begin{proof}

  First, observe that a Little bump and the Coxeter-Knuth move
  $\beta_k$ will only interact if one of the pushes in the bump is
  applied to an entry in the window $[k,k+3]$.  The inversions
  introduced by $a_i$ and $\beta_k(a)_i$ are the same when $i\not \in
  [k,k+3]$.  Therefore, since $a$ and $\beta_k(a)$ are reduced words
  of the same permutation, we see the inversions introduced by $a_k
  a_{k+1} a_{k+2} a_{k+3}$ in $a$ and $\beta_k(a_k a_{k+1} a_{k+2}
  a_{k+3})$ in $\beta_k(a)$ are the same as well.  Therefore if a
  crossing with index in $[k,k+3]$ is pushed when a Little bump
  $B^{\delta}_{(i,j)}$ is applied to $a$, such a crossing will also be
  pushed when $B^\delta_{(i,j)}$ is applied to $\beta_k(a)$, though
  not necessarily the same position.  Our argument relies on showing
  commutation can be reduced to a local check of how $\beta_k$
  interacts with $B^\delta_{(i,j)}$.  In particular, the result will
  follow from establishing two properties:
\begin{enumerate}

\item $B^\delta_{(i,j)} (a)$ and $B^\delta_{(i,j)}(\beta_k(a))$
 also differ by a Coxeter-Knuth move at position $k$.

\item The final push to a swap acted on by the Coxeter-Knuth move
has the same effect on $w$ for both $a$ and $\beta_{k}(a)$, hence would
introduce the same defect should the bump continue.

\end{enumerate}

When the entries acted on by a Coxeter-Knuth move differ by two or
more, these properties are trivial to confirm.  For entries that are
closer, the features can be checked for each type of Coxeter-Knuth
move either by hand or by computer program.  There are as many as four
checks for each type of Coxeter-Knuth move, depending on the initial
inversion and whether $0$ appears in the word.  These can be performed
by verifying the result for all possible bumps on reduced words in $B_5$ of
length 4. See Figure~\ref{fig:commute example} for example. \end{proof}

\begin{figure}
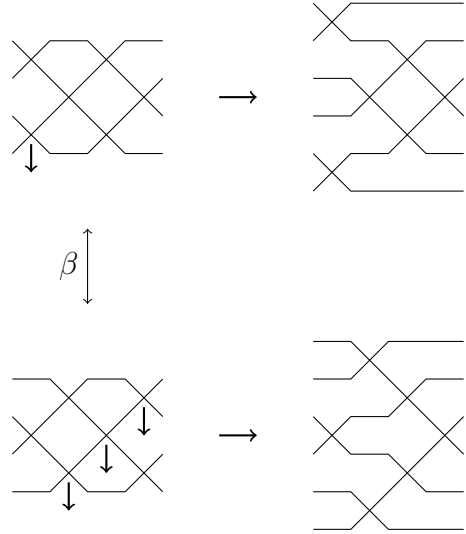

\caption{$B^{-}_{(1,2)}(1010)$ compared to $B^{-}_{(1,2)}(0101)$. The results differ by a Coxeter-Knuth move and the final pushes introduce the same transposition.
\label{fig:commute example}}
\begin{center}

\vspace{10pt}
\pic[scale = .5]

\draw (0,0) -- (3,3) -- (4,3);
\draw (0,1) -- (1,0) -- (2,0) -- (4,2);
\draw (0,2) -- (1,3) -- (2,3) -- (4,1);
\draw (0,3) -- (3,0) -- (4,0);

\draw[thick, ->] (.5,.25) -- (.5,-.5);

\draw[thick, ->] (5.5,1.5) -- (6.5,1.5);

\draw (12,-5) -- (10,-5) -- (9,-6) -- (8,-6);
\draw (12,-6) -- (11,-6) -- (10,-7) -- (9,-7) -- (8,-8);
\draw (12,-7) -- (9,-10) -- (8,-10);
\draw (12,-8) -- (9,-5) -- (8,-5);
\draw (12,-9) -- (11,-9) -- (10,-8) -- (9,-8) -- (8,-7);
\draw (12,-10) -- (10,-10) -- (9,-9) -- (8,-9);

\draw[<->] (2,-2) -- (2,-4);
\node at (1.5,-3) {$\beta$};

\draw (0,-6) -- (1,-6) -- (4,-9);
\draw (0,-7) -- (2,-9) -- (3,-9) -- (4,-8);
\draw (0,-8) -- (2,-6) -- (3,-6) -- (4,-7);
\draw (0,-9) -- (1,-9) -- (4,-6);

\draw[thick, ->] (3.5,-6.75) -- (3.5,-7.5);
\draw[thick, ->] (2.5,-7.75) -- (2.5,-8.5);
\draw[thick, ->] (1.5,-8.75) -- (1.5,-9.5);

\draw[thick, ->] (5.5,-7.5) -- (6.5,-7.5);

\draw (12,-1) -- (9,-1) -- (8,0);
\draw (12,0) -- (11,0) -- (9,2) -- (8,2);
\draw (12,1) -- (10,3) -- (9,3) -- (8,4);
\draw (12,2) -- (10,0) -- (9,0) -- (8,-1);
\draw (12,3) -- (11,3) -- (9,1) -- (8,1);
\draw (12,4) -- (9,4) -- (8,3);

\epic
\end{center}
\end{figure}

Notice that if we weakly order all shifted standard tableaux of shape
$\lambda$ by their peak sets in lexicographical order, then the unique
maximal element $U_\lambda$ is obtained by placing $1$ through
$\lambda_1$ in the first row, $\lambda_1+1$ through $\lambda_1+\lambda_2$ in the second row, and so
on. Further notice that $peaks(U_\lambda)=\{\lambda_1,
\lambda_1+\lambda_2, \lambda_1+\lambda_2+\lambda_3, \ldots \}$.

\begin{lem}
\label{t:Q.preserving2}
Let $w$ be a signed permutation, $a \in R(w)$, $B^{\delta}_{(i,j)}$ be
a Little bump for $w$ and $b = B^{\delta}_{(i,j)}(a)$.  Then
$Q^\prime(a) = Q^\prime(b)$.
\end{lem}

\begin{proof}
 We first show that $Q^\prime(a)$ and $Q^\prime(b)$ have the same
shape.  By Lemma~\ref{lem:same.peaks} and Lemma~\ref{lem:peaks}, 
\[
peaks(Q'(a))=peaks(a)=peaks(b)=peaks(Q'(b)). 
\]
Let $a'$ and $b'$ be the reduced words with maximal peak sets 
in the Coxeter-Knuth class of $a$ and $b$, respectively. 
Applying Lemma~\ref{lem:ck-little.commute}, we may assume that $a=a'$.
The shape of $Q'(a')=U_\lambda$ is determined by its peak set.  Hence, the shape of $Q'(b)$ must
be at least as large as the shape of $Q'(a)$ in dominance order.  By
assuming that $b=b'$, we can conclude the converse. Hence, $Q'(a)$ and $Q'(b)$ have the same shape $\lambda$. Furthermore, $Q'(a')=U_\lambda=Q'(b')$.


We now proceed to showing that $Q'(a)=Q'(b)$.
By Theorem~\ref{t:kras}, there exists a sequence of
Coxeter-Knuth moves
$\beta=\beta_{i_1}\circ\beta_{i_2}\circ\cdots\circ\beta_{i_k}$ such
that $\beta(a') =a$.  From Lemma~\ref{prop: beta on Q}, we see
\[
Q'(a) = Q'(\beta(a')) = h_{i_1} \dots h_{i_k} (U_\lambda).
\]
Applying Lemma~\ref{lem:ck-little.commute}, $\beta(b')=b$, so
\[
Q'(b) = Q'(\beta(b')) = h_{i_1} \dots h_{i_k} (U_\lambda),
\]
from which we conclude that $Q'(a) = Q'(b) = Q'(B^\delta_{(i,j)}(a))$.
\end{proof}

As a consequence of Lemma~\ref{t:Q.preserving2}, we prove an analog of Thomas Lam's conjecture for signed permutations.
Two reduced words $a$ and $b$ \emph{communicate} if there exists a sequence of Little bumps $B^{\delta_1}_1, B^{\delta_2}_2, \dots, B^{\delta_n}_n$ such that $b = B^{\delta_n}_n(\dots B^{\delta_1}_1(a))$.
Since Little bumps are invertible, this defines an equivalence relation.

\begin{thm}[Restatement of Theorem~\ref{t:Q.preserving}, part 2]
\label{prop:lam.conj}
Let $a$ and $b$ be reduced words. Then $Q'(a) = Q'(b)$ if and only if they communicate via Little bumps.

\end{thm}

\begin{proof}
If $a$ and $b$ communicate, then we
see $Q'(a) = Q'(b)$  by Lemma~\ref{t:Q.preserving2}.  Therefore we only need to prove the converse.  

Let $Q = Q'(a) = Q'(b)$.  We show that $a$ and $b$ both communicate
with some reduced word $c$ uniquely determined by $Q$.  Since
communication is an equivalence relation, this will complete our
proof.  Recall from Theorem~\ref{thm:transition.B} that by repeated
application of the transition equations, we may express any C-Stanley
symmetric function as the sum of C-Stanley symmetric functions of
increasing signed permutations.  Since canonical Little bumps follow the
transition equations by Corollary~\ref{c:little.bijection}, repeated applications of canonical Little bumps
will transform any reduced word $a$ into some reduced word $c$ of an
increasing permutation $u$.  Since $a$ and $c$ communicate, $Q'(a) =
Q'(c)$.  By Equation~\eqref{eq:increasing} and the fact that $u$ is
increasing, $F_{u}(X) = Q_\mu(X)$ for some $\mu$ and the reduced
expressions for $u$ are in bijection with the standard tableaux of
shifted shape $\mu$ under Kra\'{s}kiewicz insertion.  This implies
that $c$ is uniquely determined by $Q'(a)$, and hence for $Q'(b)$ as well.  
Therefore, every reduced word $a$ with $Q'(a) = Q$ communicates
with the same word $c \in R(u)$.
\end{proof}

\bigskip

\begin{cor}\label{cor:unique.min.red.word}[Restatement of Theorem~\ref{t:Q.preserving}, part 3] 
Every communication class under signed Little bumps has a unique reduced word for an increasing
signed permutation. 
\end{cor}

For permutations, Theorem 3.32 in~\cite{lam1995b} shows that  Kra\'skiewicz insertion coincides with Haiman's shifted mixed insertion.
From this, we can conclude the following.

\begin{cor}\label{cor:arb.distinct.gens}
  Let $\{j_1< j_2<...< j_p\}$ be an increasing sequence of $p$ distinct non-negative
  integers.  Every communication class containing words of length $p$
  under signed Little bumps contains a reduced word that is a
  permutation of $\{j_1, j_2,..., j_p\}$.
\end{cor}

This result can also be proved using Little bumps.

%

 \section{Axioms for shifted dual equivalence graphs}\label{s:axioms}

In this section, we build on the connection between shifted dual
equivalence operators $h_i$ and type B Coxeter-Knuth moves $\beta_i$
as stated in Lemma~\ref{prop: beta on Q}.  In particular, we define
and classify the shifted dual equivalence graphs associated to these
operators via two local properties.  Along the way, we also
demonstrate several important properties of these graphs. The approach
is analogous to the axiomatization of dual equivalence graphs by Assaf
\cite{assaf2010dual}, which was later refined by Roberts \cite{roberts2013dual}.

\begin{defn}\label{d:sdegs}
Fix a strict partition $\lambda \vdash n$.  By definition, $h_i$ acts
as an involution on the standard shifted tableaux of shape $\lambda$,
denoted $SST(\lambda )$.  Given $\lambda$, define the \textit{standard
shifted dual equivalence graph of degree} $n$ for $\lambda$, denoted
\[
\msg_\lambda= (V,\sigma, E_{1}\cup \ldots \cup E_{n-3})
\]
as follows.  The vertex set $V$ is $SST(\lambda )$, and the labeled edge sets
$E_{i}$ for $1 \leq i \leq n-3$ are given by the nontrivial orbits of
$h_i$ on $SST(\lambda )$.  To define the signature $\sigma$, recall from
Section~\ref{s:background} that every tableau $T \in SST(\lambda )$
has a peak set, denoted $peaks(T)$.  We encode a peak set by a
sequence of pluses and minuses denoted $\sigma(T) \in \{+,-\}^{n}$,
where $\sigma_{i}(T)= +$ if and only if $i$ is a peak in $T$.  We
refer to $\sigma (T)$ as the \textit{signature} of $T$.  Note that
peaks never occur in positions 1 or $n$ and that they never occur
consecutively.  Conversely, any subset of $[n]$ that satisfies these
properties is the peak set of some tableau, hence we will call it an
\textit{admissible peak set}.
\end{defn}


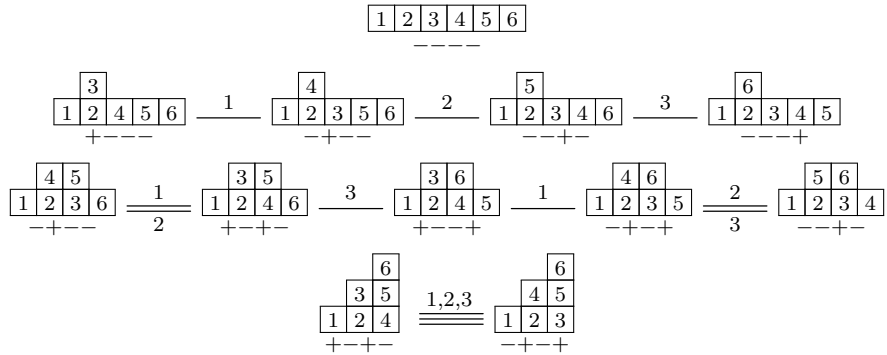
\begin{figure}[h]
\ytableausetup{aligntableaux=center, smalltableaux}
\[ \begin{array}{c}
  \vcenter{\vbox{
\xymatrix{ {\ytableaushort{123456} \atop ----}
}}}
 \\
\vcenter{\vbox{ \xymatrix{
		{\ytableaushort{\none3,12456} \atop +---} \ar@{-}[r]^{1}  & {\ytableaushort{\none4,12356}\atop -+--} \ar@{-}[r]^{2}	& {\ytableaushort{\none5,12346}\atop --+-} \ar@{-}[r]^3   & {\ytableaushort{\none6,12345}\atop ---+}
} }}
\\
\vcenter{\vbox{
	\xymatrix{
		 {\ytableaushort{\none45,1236} \atop -+--} \ar@{=}[r]^{1}_{2} & {\ytableaushort{\none35,1246}\atop +-+-} \ar@{-}[r]^{3}	& {\ytableaushort{\none36,1245}\atop +--+} \ar@{-}[r]^1	& {\ytableaushort{\none46,1235}\atop -+-+} \ar@{=}[r]^2_3 & {\ytableaushort{\none56,1234} \atop --+-} 
		 }}}
\\
\vcenter{\vbox{ \xymatrix{	
	{\ytableaushort{\none\none6,\none35,124}\atop +-+-}\; \ar@3{-}[r]^{1,2,3}	& {\ytableaushort{\none\none6,\none45,123}\atop -+-+} 
		} } } 
\end{array} \]
\caption{The standard shifted dual equivalence graphs of degree 6 omitting $\sigma_{1}$ and $\sigma_{6}$.
\label{fig: proto SDEG}
}
\end{figure}

In Figure~\ref{fig: proto SDEG}, all of the standard shifted dual
equivalence graphs of degree 6 are drawn and labeled by their
signatures omitting $\sigma_{1}$ and $\sigma_{6}$ since $1$ and $6$
can never be in an admissible peak set.  Already from this figure we
can see that standard shifted dual equivalence graphs are not always
dual equivalence graphs because they can have two vertices connected
by 3 edges labeled $i,i+1,i+2$.  Also, observe that if vertices $v$
and $w$ are contained in an $i$-edge, then
$\sigma_{[i+1,i+2]}(v)=-\sigma_{[i+1,i+2]}(w)$. Furthermore, notice
that the label $i$ of the edge and whether or not it is a double or
triple edge can be determined entirely from the peak sets.  This fact
will be used often in the proofs that follow.  From this figure, we
also can determine all of the possible standard shifted dual
equivalence graphs for $n=1,2,3,4,5$ by fixing the values higher than
$n$.

The standard shifted dual equivalence graphs have several nice
properties on par with the dual equivalence graphs or equivalently
Coxeter-Knuth graphs of type $A$.  By Theorem~\ref{thm: Haiman}, each
$\msg_\lambda$ is connected.  Observe that by
Definition~\ref{def:schur.Q}, a Schur $Q$-function $Q_\lambda$ is the
generating function for the sum of peak quasisymmetric functions
associated to the labels on the vertices of $\msg_\lambda$.  Recall from
Section~\ref{s:Kraskiewicz} that the lexicographically largest peak
set for all standard shifted tableaux of a fixed shape $\lambda$ is
given by the unique tableau $U_\lambda$.  Thus, the shape $\lambda$
can be recovered from the multiset of signatures on the vertices.

Each standard shifted dual equivalence graph is an example of the
following more general type of graph.

\begin{defn}
  Let $C$ and $S$ be two finite ordered lists.  An $S$-\emph{signed},
    $C$-\emph{colored graph} consists of the following data:
\begin{enumerate}\itemsep-.0in
\vspace{-.02in}
\item a finite vertex set $V$,
\item a signature function $\sigma:V \rightarrow \{+,- \}^{|S|}$
associating a subset of $S$ to each vertex,
\item a collection $E_i$ of unordered pairs of distinct vertices in $V$ for each  $i \in C$.  
\end{enumerate}
\noindent 
A signed colored graph 
is denoted $\mg =(V, \sigma, E)$
where $E=\cup_{i \in C}E_i$.  We 
say that $\mg$ has \emph{shifted
degree n} if $C=[n-3]$, $S=[n]$ and $\sigma(v)$ is an admissible peak
set for an integer sequence of length $n$ for all $v\in V$.  
The signature $\sigma(v)$ is encoded by a sequence in $\{+,- \}^{n}$ where
$\sigma_{i}(v)=+$ if and only if $i \in \sigma(v)$.  
We use the
notation $\sigma_{[i,j]}(v)$ to mean the subset $\sigma(v) \cap [i,j]$
which can be encoded by $+$'s and $-$'s as well. 
\end{defn}

\begin{defn}
  Given two $S$-signed $C$-colored graphs $\mg=(V,\sigma,E)$ and
  $\mg^\prime=(V^\prime, \sigma^\prime, E^\prime)$, a \emph{morphism
    of signed colored graphs} $\phi\colon G\to G'$ is a map from $V$
  to $V^\prime$ that preserves the signature function and induces a
  map from $E_i$ into $E^\prime_i$ for all $i \in C$. An
  \emph{isomorphism} is a morphism that is a bijection on the vertices
  such that the inverse is also a morphism.
\end{defn}

\begin{defn}
  A signed colored graph $\mg$ is a \emph{shifted dual
    equivalence graph} \textup{(SDEG)} if it is isomorphic to a
  disjoint union of standard shifted dual equivalence graphs.
\end{defn}

The next lemma allows us to classify the isomorphism type of any
connected SDEG by a unique standard shifted dual equivalence graph.

\begin{lem}\label{lem:triv automorphism}
  Let $\msg_\lambda$ and $\msg_\mu$ be any two standard shifted dual
  equivalence graphs. If $\phi\colon \msg_\lambda \to \msg_\mu$ is an
  isomorphism, then $\lambda=\mu$ and $\phi$ is the identity map.
\end{lem}
\begin{proof} Suppose that $\phi: \msg_\lambda \to \msg_\mu$ is an
isomorphism. Then the vertices of $\msg_\lambda$ and $\msg_\mu$ must
have the same multisets of associated peak sets. By looking at the
unique lexicographically maximal peak set in both, it follows that
$U_\lambda=U_\mu$. In particular, $\lambda=\mu$. Thus, $\phi$ is an
automorphism that sends $U_\lambda$ to itself.  Since $h_i$
defines the $i$-edges in both $\msg_\lambda$ and $\msg_\mu$ and both
graphs are connected, we see that  $\phi$ acts as the identity map.
\end{proof}

The connection between shifted dual equivalence graphs and the type B
Coxeter-Knuth graphs stated in Theorem~\ref{t:ckg} is now readily
apparent.  Recall, Theorem~\ref{t:ckg} states that every type B
Coxeter-Knuth graph $CK_B(w)$ with signature function given by peak
sets is a shifted dual equivalence graph, where the isomorphism is
given by the Kra\'{s}kiewicz $Q'$ function. It further states that
every shifted dual equivalence graph is also isomorphic to some
$CK_B(w)$. We give the proof of this theorem now.

\begin{proof}[{\bf Proof of Theorem~\ref{t:ckg}}] We show
that the map $Q'$ sending vertices in $CK_B(w)$ to their recording
tableaux is the desired isomorphism. This follows immediately from the
definition of the Kra\'{s}kiewicz insertion algorithm,
Lemma~\ref{lem:same.peaks} and Lemma~\ref{prop: beta on Q}.

  To see the converse statement, observe that $\msg_\mu$ is isomorphic
  to the Coxeter-Knuth graph $CK_B(w)$ for the increasing signed
  permutation $w=w(\mu)$ as defined just before Equation~\ref{eq:increasing}.
\end{proof}



\begin{defn}\label{def:sdeg.restriction}
  Given a signed colored graph $\mg=(V, \sigma, E)$ of shifted degree
$n$ and an interval of nonnegative integers $I=[a,b]\subset [n]$,
let 
\[
\mg^{I}=\left(V, \sigma, E_{a}\cup E_{a+1}\cup \dots \cup E_{b-3} \right)
\]
 denote the subgraph of $\mg$ using only the $i$-edges for $a\leq
i\leq b-3$.  Also define the \textit{restriction of $\mg$ to $I$}, to be
the signed colored graph
\[
\mg|_I= (V, \sigma^{\prime}, E^{\prime})
\]
\begin{enumerate}
\item 
$\sigma^\prime(v)=\{s-a+1 \mid s \in \sigma(v)\cap (a,b)\}$, 
\item
$E_i^\prime= E_{a+i-1}$ when $i \in [|I|-3]$.  
\end{enumerate}
\end{defn}

\noindent Notice that the vertex sets of $\mg$, $\mg^{I}$ and $\mg|_I$
are the same.  If $\mg$ is a signed colored graph with shifted degree
$n$ and $I=[a,b]$ then $\mg|_{I}$ will have shifted degree, but the
degree will be at most $|I|$. It could be strictly less than $|I|$ if
$n<b$.

Recall the two desirable properties of a signed colored graph $\mg$
stated in Theorem~\ref{SDEG properties}.  We name them here so we can
refer to them easily.  

\begin{enumerate}
\item \textbf{Locally Standard}: If $I$ is an interval of positive integers
with $|I|\leq 9$, then each component of $\mg|_I$ is isomorphic to a
standard shifted dual equivalence graph of degree up to $|I|$.
\item \textbf{Commuting}: If $(u, v) \in E_{i}$ and $(u,w) \in E_{j}$
then there exists a vertex $y \in V$ such that $(v,y) \in E_{j}$
and $(w,y) \in E_{i}$. Thus the components of $(V,E_{i} \cup E_{j})$
for $|i-j|>3$ are commuting diamonds.
\end{enumerate}

\begin{lem}\label{lem:properties hold for SDEGS}
  For any standard shifted dual equivalence graph $\msg_\lambda$, both
the Locally Standard Property and the Commuting Property hold.  In
fact, $\msg_\lambda|_{I}$ is an SDEG for all intervals $I$.  
\end{lem}
\begin{proof}
  Consider a standard shifted dual equivalence graph $\msg_\lambda$ for $\lambda \vdash n$.  The
  Commuting Property must hold because $h_i$ acts according to the
  positions of the values in $[i,i+4]$ only. Hence, $h_i$ and $h_j$
  commute provided $|i-j|>4$.

To demonstrate the Locally Standard Property for a given interval $I$,
observe that we can restrict any $T\in SST(\lambda )$ to the values in
$I$ which form a skew shifted tableau and all the data for
$\msg_\lambda|_{I}$ will still be determined.  By Lemma~\ref{lem:
slides and h commute}, jeu de taquin slides commute with the
$h_{i}$'s.  So the isomorphism from $\msg_\lambda|_{I}$ to a union of
standard shifted dual equivalence graph is given by restriction and
repeated application of the jeu de taquin operator $\Delta$ defined in
Definition~\ref{def:delta}.
\end{proof}


We note that it is also straightforward to prove
Lemma~\ref{lem:properties hold for SDEGS} by appealing to the fact
that $\msg_\lambda$ is isomorphic to the Coxeter-Knuth graph $CK_B(w)$
for the increasing signed permutation $w=w(\lambda)$.  We know the
$\beta_{i}$'s satisfy the Commuting Property.  Furthermore,
restriction on a Coxeter-Knuth graph gives rise to an isomorphism with
another Coxeter-Knuth graph since every consecutive subword of a
reduced word is again reduced.  It is instructive for the reader to
consider the alternative proof for the lemmas below using
Coxeter-Knuth graphs if that language is more familiar.

\begin{lem}\label{ax6}
Given a strict partition $\lambda$ of size $n$, any two distinct
components $\ma$ and $\mb$ of $\msg_\lambda^{[n-1]}$ are connected by
an $(n-3)$-edge in $\msg_\lambda$.  In particular, any two vertices in
$\msg_\lambda$ are connected by a path containing at most one
$(n-3)$-edge that is not doubled by an $(n-4)$-edge. 
\end{lem}
\begin{proof}
It follows from Theorem~\ref{thm: Haiman} that $\ma$ and $\mb$ are
characterized by the position of $n$ in their respective shifted
tableaux. Suppose $\ma$ and $\mb$ have $n$ in corner cells $c$ and
$d$, respectively, with 
$c$ in a lower row than $d$.  Then there exist tableaux $S \in \ma$
and $T \in \mb$ that agree everywhere except in the cells containing
$n-1$ and $n$ such that $n-2$ lies between $n$ and $n-1$ in the
reading word of $S$ and $n-3$ comes before $n-2$.  Thus, by definition
of $h_{n-3}$, we have $h_{n-3}(S)=T$, so $\ma$ and $\mb$ are connected
by an $(n-3)$-edge.  
This edge cannot be an $(n-4)$-edge since $\ma$ and $\mb$ are not connected in $\msg_\lambda^{[n-1]}$. 
\end{proof}


\begin{lem}\label{lem: sig window}
Let $\mg=(V,\sigma ,E)$ be a signed colored graph of shifted degree
$n$ satisfying the Commuting Property and the local condition that
$\mg|_{[j,j+5]}$ is a shifted dual equivalence graph for all $1\leq
j\leq n-5$.  If $v,w\in V$ are connected by an $i$-edge in $\mg$, then
$\sigma_{k}(v)=\sigma_{k}(w)$ for all $k \not \in [i-1,i+4]$.
\end{lem}

\begin{proof}
The lemma clearly holds for standard SDEGs by the definition of the
shifted dual equivalence moves $h_{i}$ which determine the $i$-edges.
For $k=i+5$, the lemma holds since $\mg|_{[i,i+5]}$ is a shifted dual
equivalence graph.  Now assume that $i+5< k\leq n$.  Say $v,w\in V$
are connected by an $i$-edge in $\mg$, and assume
$\sigma_{j}(v)=\sigma_{j}(w)$ for all $i+5 \leq j <k$ by induction.  By
the local condition, the vertex $v$ admits an $(k-2)$-edge if and only
if $\sigma_{k-1}(v)=+$ or $\sigma_{k}(v)=+$. 
These possibilities are exclusive
since the signature encodes an admissible peak set.  Thus,
$\sigma_{k}(v)$ is determined by $\sigma_{k-1}(v)$ and the presence or
absence of an adjacent $(k-1)$-edge.  Since $i$-edges and
$(k-2)$-edges commute for $k-2 - i \geq 4 $ by the Commuting Property,
we know that $v$ admits a $(k-2)$-edge if and only if $w$ admits a
$(k-2)$-edge.  Since $\sigma_{k-1}(v)=\sigma_{k-1}(w)$ we obtain
$\sigma_{k}(v)=\sigma_{k}(w)$ by the same considerations.  Therefore,
recursively $\sigma_{k}(v)=\sigma_{k}(w)$ for all $i+4<k \leq n$.

A similar argument works for all $1\leq k <i-1$.  This completes the proof.
\end{proof}

\begin{lem}\label{lem: injective iso}
  Let $\lambda$ be a strict partition of $n$ and $\mg=(V,\sigma ,E)$
be a signed colored graph of shifted degree $n$ satisfying the Locally
Standard and Commuting Properties.  If $\phi: \mg \longrightarrow
\msg_{\lambda}$ is an injective morphism, then it is an isomorphism.
\end{lem}

  \begin{proof} Let $v \in V$ and say $\phi (v)=T \in SST(\lambda)$.
Since $\phi$ is signature preserving and $\mg$ is Locally Standard,
we can apply Lemma~\ref{lem: sig window}  to show that $v$ has an
$i$-neighbor in $\mg$ if and only if $T$ has an $i$-neighbor in
$\msg_{\lambda}$ and a similar statement holds for each of their
neighbors.  Furthermore, since $\phi$ is an injective morphism $(v,w)
\in E_{i} \cap E_{j}$ if and only if $h_{i}(T)=h_{j}(T)=\phi (w)$.
Thus, $\phi$ induces a bijection from the neighbors of $v$ to the
neighbors of $T$ that preserves the presence or absence of
$i$-neighbors.  In particular, every neighbor of $T$ in
$\msg_{\lambda}$ is in the image of $\phi$.  Since $\msg_{\lambda}$ is
connected, there is a path from $T$ to any other vertex $S$ in
$\msg_{\lambda}$ and by iteration of the argument above we see that
$\phi$ maps some vertex in $V$ to $S$.  Hence, $\phi$ is both
injective and surjective on vertices, and the inverse map is also a
morphism of signed colored graphs.  Thus, $\phi$ is an isomorphism.
\end{proof}

With Lemma~\ref{lem: injective iso} in mind, our goal will be to
demonstrate the existence of an injective morphism from any connected signed
colored graph satisfying the Locally Standard and Commuting Properties
to a standard SDEG. To do this, we
will employ an induction on the degree of the signed colored graphs in
question. The next lemma is an important part of that induction.

\begin{lem}\label{lem: unique extension} 
  Let $\mg=(V,\sigma, E_1 \cup \ldots \cup E_{n-2})$ be a signed
colored graph of shifted degree $n+1$ that satisfies the following
hypotheses.
\begin{enumerate}
\item The Commuting Property holds on all of $\mg$.  
\item Both $\mg|_{[n]}$ and $\mg|_{[n-6,n+1]}$ are shifted dual
equivalence graphs.
\end{enumerate}
Let $\mc$ be a component of $\mg^{[n]}$.  Then the following
properties hold:
\begin{enumerate}
\item There exists a unique strict partition $\mu$ of degree $n+1$ and a
signed colored graph isomorphism $\phi$ mapping $\mc$ to a
component of $\msg_\mu^{[n]}$.
\item For every vertex $v\in V(\mc)$,\ $v$ has an $(n-2)$-neighbor in
$\mg$ if and only if $\phi(v)$ has an $(n-2)$-neighbor in $\msg_\mu$.
\end{enumerate}
\end{lem}

We refer to $\msg_\mu$ in this lemma as the \emph{unique extension} of
$\mc$ in $\mg$.  The outline of this proof is based on the proof of
Theorem~3.14 in \cite{assaf2010dual}, but it uses peak sets in
addition to ascent/descent sets for tableaux.

\begin{proof}
By hypothesis, $\mc|_{[n]}$ is isomorphic to $\msg_\lambda$ for some
strict partition $\lambda \vdash n$, so we can bijectively label the
vertices of $\mc$ by standard shifted tableaux of shape $\lambda$ in a
way that naturally preserves the signature functions $\sigma_{i}$ for
all $1 \leq i<n$.  Since $\mg|_{[n-6,n+1]}$ is an SDEG, the lemma is
automatically true if $n\leq 7$, so assume $n>7$.

Partition the vertices of $\mc$ or equivalently $SST(\lambda)$
according to the placement of $n-1$ and $n$.  Let $D_{ij}$ be the
subgraph of $\mc$ with $n$ in row $i$ and $n-1$ in row $j$ with edges
in $E_1 \cup \ldots \cup E_{n-5}$, then each $D_{ij}$ is connected
since its restriction to $[n-2]$ is also isomorphic to a standard SDEG by
hypothesis.    Similarly, let $D_i$ be the connected subgraph of $\mc$
with vertex set labeled by tableaux with $n$ in row $i$ along with the
corresponding edges in $E_1 \cup \ldots \cup E_{n-4}$.

We first show that $\sigma_n$ is constant on $D_{ij}$. 
By Lemma~\ref{ax6}, each pair $ S,T \in V(D_{ij})$ 
may be connected by a path using only edges in $E_1\cup\dots\cup E_{n-5}$.
By Lemma~\ref{lem: sig window}, 
$\sigma_n$ is constant on $D_{ij}$.  The same fact 
need not hold for 
the $D_{i}$. 
%
%
We will show that there is a unique row of $\lambda$ 
where  $n+1$ can be placed  that is simultaneously consistent 
with the signatures for all vertices in all the $D_i$'s. 
 This placement must also be consistent with the
  existence of $n-2$-edges in $\mg$, completing the proof.

We proceed by partitioning the $D_{ij}$ for a fixed $i$ into three
types and describing how to extend each type consistently.  First,
suppose that there is some nonempty $D_{ij}$ with $i>j$. Then $n$ is
in a strictly higher row than $n-1$ in all the tableaux labeling
vertices of $D_{ij}$.  Furthermore 
there is some tableau $T$ labeling a vertex of
$D_{ij}$ such that $n-2$ lies in a row weakly above row $j$ making
position $n-1$ a peak.  This implies $\sigma_n(T)=-$ since peaks
cannot be adjacent.  Since $\sigma_n$ is constant on $D_{ij}$, we see
that $i>j$ implies $\sigma_n(T)=-$ for all tableaux labeling vertices
in $D_{ij}$. Furthermore, any placement of $n+1$ will be consistent
with the fact that $\sigma_n=-$.

Second, suppose that $i\leq j$ and that $\sigma_n(T)=+$ for 
all  $T
\in V(D_{ij})$.  
We would like to add
$n+1$ to a row strictly above $i$, but we must show this will be
consistent with each neighboring component $D_{ij'}$.  
By Lemma~\ref{ax6}, the component
$D_{ij}$ is connected to every other $D_{ij'}$ by an $(n-4)$-edge.  
Such an edge $e=(T,U) \in V(D_{ij}) \times
V(D_{ij'})$ could be part of a triple edge with an $(n-3),(n-2)$-edge.
In this case, we must have $\sigma_{[n-2,n]}(T)=+-+$ and
$\sigma_{[n-2,n]}(U)=-+-$, as demonstrated in Figure~\ref{fig: proto
SDEG}.  Thus, if $U$ is a vertex in $D_{ij'}$, then position $n-1$
must be a peak of $U$ and $i>j'$.  Therefore, if $n+1$ is added in any
row to a tableau $T \in V(D_{ij'})$ it will not create a peak in
position $n$.  On the other hand, if $D_{ij}$ is connected to another
nonempty $D_{ij'}$ by an $(n-4)$-edge that is not also a $(n-2)$
edge, then again by Figure~\ref{fig: proto SDEG} one observes that
$\sigma_n(U)=+$ for all $U \in V(D_{ij'})$.  Thus, we can consistently
extend each vertex in $D_{i}$ by placing $n+1$ in such a way that it
creates a descent from $n$ to $n+1$. Any row strictly above row $i$
will work provided it results in another shifted shape.  

Third, suppose that there exists a nonempty $D_{ij}$ such that $i\leq
j$ and $\sigma_n(T)=-$ for all $T \in V(D_{ij})$.  The component
$D_{ij}$ is connected to every other $D_{ij'}$ by an $(n-4)$-edge.
Assume such an edge $e=(T,U) \in V(D_{ij}) \times V(D_{ij'})$ is part
of a triple edge with an $(n-3),(n-2)$-edge.  In this case, we must
have $\sigma_{[n-2,n]}(T)=-+-$ since $\mg|_{[n-6,n+1]}$ is an SDEG.
Thus, $n-1$ is a peak of
$T$, but this contradicts the assumption that $i\leq j$.  
Therefore no
$(n-4)$-edge 
containing a vertex in $D_{ij}$ can be part of a triple
edge with an $(n-2)$-edge.  By observing Figure~\ref{fig: proto SDEG}
again, we conclude that $\sigma_n=-$ on all of $D_{ij'}$. In this
case, we can consistently extend all tableaux labeling vertices in
$D_{i}$ by placing $n+1$ in any row weakly lower than $n$.  


We complete the proof 
by placing $n+1$ 
in a unique row $m$ consistent with the required ascents and descents in all
$D_i$. 
Let $X$ be the union of all nonempty $D_i$ containing a vertex $T$
with some $\sigma_n(T)=+$, and let $Y$ be the union of all nonempty
$D_i$ with no vertex $T$ such that $\sigma_n(T)=+$.  
Every
vertex in $X$ needs a descent from $n$ to $n+1$, and every vertex in
$Y$ needs an ascent from $n$ to $n+1$. 
To do this, 
let $m$ be the minimal positive integer such that
$i<m$ for all $D_i \in X$.
 We will show that
 $Y$ consists of all $D_i$ with $i\geq m$.

%
%


If $X$ is empty, then we may let $m=1$ and extend $\lambda$ to a
strict partition $\mu$ by adding one box to the first row of
$\lambda$.  Then $\mc|_{[n]}$ is isomorphic to the component of
$\msg_\mu|_{[n]}$ with $n+1$ fixed in the first row and the
conclusions of the lemma hold.   

Assume $X$ is nonempty 
and that $i<j$ exist such that $D_i$ and $D_{j}$ are nonempty with $D_{i} \in Y$.  Since $\mc$
is connected, there exists an $(n-3)$-edge $e=(S,T)$ with $S \in D_i$
and $T\in D_{j}$.  
By the definition of shifted dual equivalence moves on
$\msg_\lambda|_{[n]}$, $e$ must be an
$(n-3)$-edge that acts as the transposition $t_{n,n-1}$ on $S$ and $T$.
This implies $S$ 
has $n-1$ in row
$j$ and $n$ in row $i$.  In this configuration, $n-1$ cannot be the
position of a peak in $S$. Thus $S$ must have a peak in position $n-2$
since it is a vertex of an $(n-3)$-edge.
If $\sigma_n(S)=+$, then it would contradict the hypothesis that
$D_{i} \in Y$.  Therefore $\sigma_{[n-2,n]}(S)=+--$, which implies
$\sigma_{[n-2,n]}(T)=-+-$. 
  In particular, $S$ and $T$ are not connected by an
$(n-2)$-edge.  We then conclude that $T$ must have an ascent from
$n$ to $n+1$.  
Thus $D_j \in Y$ for all $j>i$. 

We conclude that $Y$ consists of all the $D_{i}$ for all $i \geq m$
and $X$ consists of all $D_{i}$ for $i<m$.  Hence, $n+1$ may placed in
row $m$, while no other choice of row can simultaneously satisfy the
required ascents and descents from $n$ to $n+1$ in all $D_{i,j}$,
completing the proof.
\end{proof}

We can also find a \textit{unique lower extension} of a component of
$\mg|_{[2,n+1]}$ provided similar conditions hold.  For the next
lemma, recall $\Delta$ from Definition~\ref{def:delta}.

\begin{lem}\label{lem: permuting ribbons}
Given two shifted standard tableau $T$ and $U$ of shifted shape
$\lambda\vdash n$, $T$ and $U$ are in the same component of
$\msg_\lambda|_{[2,n]}$ if and only if $\Delta(T)$ and $\Delta(U)$
have the same shape.
\end{lem}
\begin{proof}
By definition, $T$ and $U$ are in the same component of
$\msg_\lambda|_{[2,n]}$ if and only if they are related by a sequence
of shifted dual equivalence moves $h_{i}$ for $2\leq i\leq n-3$.
Lemma~\ref{lem: slides and h commute} implies that $\Delta\circ h_i = h_{i-1}\circ \Delta$,
so $T$ and $U$ are in the same
component if and only if $\Delta (T)$ and $\Delta(U)$ are related by a
sequence of shifted dual equivalence moves $h_{i}$ for $1\leq i\leq
n-4$.  
By Theorem~\ref{thm: Haiman}, $\Delta (T)$ and $\Delta(U)$ are
related by a sequence of shifted dual equivalence moves $h_{i}$ for
$1\leq i\leq n-4$ if and only if they have the same shape.
\end{proof}

\begin{lem}\label{lem: shift [2,n]}
 Let $\mg=(V,\sigma, E_{1} \cup \dots \cup E_{n-3})$ be a connected
signed colored graph of shifted degree $n$ satisfying the Commuting
Property such that $\mg|_{[n-1]}$ and $\mg|_{[2,n]}$ are SDEGs.  Let
$\mc$ be any component of $\mg|_{[n-1]}$, and let $\msg_\mu$ be the
unique extension of $\mc$ in $\mg$.  Let $v \in V(\mc)$, and let
$\mc'$ be the component of $v$ in $\mg|_{[2,n]}$.  Say $\msg_\lambda
\cong \mc'$.  If $v$ is mapped to $T \in SST(\mu)$ in $\msg_\mu$, then
$v$ is mapped to $\Delta (T)$ in $\msg_\lambda$.
\end{lem}

\begin{proof}
The proof follows from Lemma~\ref{lem: unique extension} and
Lemma~\ref{lem: permuting ribbons}.
\end{proof}

\begin{lem}\label{lem: correct components}
 Let $\mg=(V,\sigma, E_{1} \cup \dots \cup E_{n-3})$ be a connected
signed colored graph of shifted degree $n>9$ satisfying the Commuting Property such that $\mg|_{[n-1]}$ and
$\mg|_{[2,n]}$ are SDEGs. 
Let $\mc$ be any component of $\mg^{[n-1]}$, and let $\msg_\mu$ be the
unique extension of $\mc$ in $\mg$.
\begin{enumerate}
\item An $(n-3)$-edge connects two vertices in $\mc$ if and only if the
corresponding vertices are connected by an $(n-3)$-edge in
$\msg_\mu$.
\item If two $(n-3)$-edges connect the image of $\mc$ to the same component
in $\msg_\mu^{[n-1]}$, then corresponding edges in $\mg$ must also
connect $\mc$ to the same component in $\mg^{[n-1]}$.
\end{enumerate}

%
%
\end{lem}
\begin{proof}
We begin by considering the case where $(u,v)\in E_{n-3}$ and $u,v \in
\mc$.  Since $\msg_{\mu}$ is the unique extension of $\mc$, we can
associate tableaux $S,T$ with $u,v$ respectively.  We want to show
$h_{n-3}(S)=T$.  By hypothesis $\mg|_{[2,n]}$ is an SDEG. The
component $\mc'$ of $\mg|_{[2,n]}$ containing $u$ is isomorphic to the
component of $S$ in $\msg_{\mu}|_{[2,n]}$ by Lemma~\ref{lem: unique
extension}.  The vertex $u$ maps to $\Delta(S)$ under this
isomorphism by Lemma~\ref{lem: shift [2,n]}.  Since $(u,v) \in E_{n-3}$, they are connected by an
$n-4$-edge in $\mc'$.  By Lemma~\ref{lem: shift [2,n]}, the image of
$v$ in $\msg_{\mu}|_{[2,n]}$ is $\Delta (T)$ and $h_{n-4}(\Delta (S))
=\Delta (T)$.  Therefore, $h_{n-3}(S) =T$ since every edge in
$\msg_{\mu}|_{[2,n]}$ comes from an edge in $\msg_{\mu}$ with one
higher label.

The previous argument is reversible. That is, given $S,T \in
SST(\mu)$ with $n$ in the same cell, if $h_{n-3}(S)=T$ then the
vertices $u,v$ in $\mc$ mapping to $S,T$ respectively must be
connected by an $(n-3)$-edge in $\mg$. This proves (1).

Next we prove (2).  By Lemma~\ref{lem: unique extension}, we can label
the vertices of $\mc$ by standard tableaux of shape $\mu$.  Let $s,t
\in \mc$ be labeled by the tableaux $S$ and $T$, respectively.
Assume that $(s,s'), (t,t') \in E_{n-3}$, but $s',t' \not \in \mc$.  
Further assume that both $S$ and $T$ are connected to the same 
component of $\msg_\mu^{[1,n-1]}$ by $(n-3)$-edges, and that this
component is distinct from the component of $T$.  Then, $n-1$ and $n$
must be in the same cells of $S$ and $T$, with $n-2$ in some cell
between the two in row reading order, and $n-3$ in some cell before
that, by the definition of $h_{n-3}$. 

If $S$ and $T$ are connected via edges in $E_1 \cup \ldots \cup
E_{n-7}$ then the lemma holds since each of these edges commutes with
edges in $E_{n-4}$.  If $S$ and $T$ are connected via edges labeled
$2, 3, \ldots, n-4$, then we can assume $s$ and $t$ are also connected
by edges in $E_2 \cup \ldots \cup E_{n-4}$ by Lemma~\ref{lem: shift
[2,n]}.  It thus suffices to show that some $S'$ in the same
$\mg^{[n-4]}$-component as $S$ and some $T'$ in the same
$\mg^{[n-4]}$-component as $T$ exist and satisfy the following
properties: both $S'$ and $T'$ admit $(n-3)$-edges that interchange
$n-1$ and $n$, and both are in the same component of
$\mc^{[2,n-1]}$. By Lemma~\ref{lem: permuting ribbons}, it suffices to
find $S'$ and $T'$ such that $\Delta(S'),\Delta(T')$ have the same
shape.

If $S|_{[n-2]}$ contains $i<n-3$ in a northeast boundary cell $c$,
then we can rearrange the entries of $S$ smaller than $i$ to get $S'$
so that the cell $c$ is moved in $\Delta(S')$ and the rest of the jeu
de taquin slides are independent of the filling. If $T|_{[n-2]}$ also
contains an entry $j<n-3$ in cell $c$, then rearrange the entries of
$T$ to agree with $S'$ in all cells weakly southwest of $c$ to obtain
$T'$.  Then the jeu de taquin process $\Delta (T')$ passes through $c$
and by construction $\Delta(S')$ and $\Delta(T')$ have the same shape
since $S$ and $T$ have the same shape and that $n-1$ and $n$ are in
the same cell in both.  Thus, $S'$ and $T'$ are connected by edges in
$E_2 \cup \ldots \cup E_{n-4}$ by Lemma~\ref{lem: permuting ribbons}.
If the shape of $S|_{[n-2]}$ has 5 or more northeast boundary cells,
then such a cell $c$ exists and the lemma holds.

There are only a finite number of strict partitions $\mu$ with at most
4 northeast boundary cells after removing 2 corner cells.  For
example, if $\mu$ has 6 or more rows or 9 or more columns then even
after removing 2 corner cells there must be at least 5 boundary cells
remaining.  We only need to consider such shapes with at least 10 cells by
hypothesis.  In each remaining case, one needs to check that no matter
how $n,n-1$ are placed in corner cells $\{c,d\}$ of $\mu$ the jeu
de taquin argument above may still be applied.  We leave the remaining
cases to the reader to check to complete the proof.
\end{proof}

\begin{figure}
\ytableausetup{aligntableaux=center, smalltableaux}
$S=$
\begin{ytableau}
\none&\none&8\\
\none&4&5\\
1&2&3&6&7&9
\end{ytableau}
\quad \quad
$T=$
\begin{ytableau}
\none&\none&8\\
\none&6&7\\
1&2&3&4&5&9
\end{ytableau}
\caption{An example from the proof of Lemma~\ref{lem: correct components} where $n=9$ and $\Delta(S)$ does not have the same shape as $\Delta(T)$.\label{fig: different delta}}
\end{figure}

\begin{lem}\label{lem: complete}
  Let $\mg=(V,\sigma,E_1 \cup \ldots \cup E_{n-3})$ be a connected signed colored graph
  of shifted degree $n> 9$ satisfying the Commuting Property such that 
  $\mg|_{[n-1]}$ is an SDEG, and $\mg|_{[2,n]}$ is an SDEG.
  Then there exists a morphism $\phi\colon \mg \rightarrow
  \msg_\lambda$ for some strict partition $\lambda \vdash n$.
\end{lem}

\begin{proof}

  Let $t$ and $u$ be distinct vertices of $\mg$ that are connected by
  an $(n-3)$-edge.  Let $\mc$ and $\md$ be the components in
  $\mg^{[n-1]}$ of these two vertices, with unique extensions
  $\msg_\lambda$ and $\msg_\mu$. Label $t$ and $u$ with $T$ and $U$,
  their tableaux in $\msg_\lambda$ and $\msg_\mu$, respectively.  It
  suffices to show that $h_{n-3}(T)=U$. In particular, this would
  guarantee that $\lambda=\mu$, and that there is a morphism from
  $\mg$ to $\msg_\lambda$.
  
  We first show that we can make three simplifying
  observations.  If $T$ and $h_{n-3}(T)$ are in the same
  component of $\msg_\lambda^{[n-1]}$, then the lemma follows from
  Lemmas~\ref{lem: correct components} and \ref{lem:triv
    automorphism}.  Now assume that $T$ and $h_{n-3}(T)$ are in different
  components of $\msg_\lambda^{[n-1]}$. By symmetry, we may also
  assume that $U$ and $h_{n-3}(U)$ are in different components of
  $\msg_\mu^{[n-1]}$.  Thus, we can assume $h_{n-3}$ acts on both $T$
  and $U$ by switching $n-1$ and $n$.

Secondly, applying Lemma~\ref{lem:triv automorphism}, the Commuting Property and
the hypothesis that $\mg|_{[n-1]}$ is an SDEG, it follows that
$T|_{[n-4]}=U|_{[n-4]}$.  We thus only need to show that
$T|_{[n-3,n]}=h_{n-3}(U)|_{[n-3,n]}$.  

For the third observation, note that the component of $T$ in
$\mg|_{[2,n]}$ is isomorphic to the component of the image of $T$ in
$\msg_\mu|_{[2,n]}$. This follows from the fact that the component of
$T$ in $\msg_\lambda|_{[2,n]}$ satisfies the definition of the unique
extension of the component of $T$ in $\mg|_{[2,n-1]}$ in
$\mg|_{[2,n]}$. Because the component of $T$ in $\mg|_{[2,n]}$ is an
SDEG, it follows from Lemma~\ref{lem:triv automorphism} that this
component is isomorphic to the component of $T$ in
$\msg_\lambda|_{[2,n]}$.  Furthermore, $T$ is taken to $\Delta(T)$ via
this isomorphism. Similarly, there is another isomorphism on the
component of $U$ in $\msg^{[2,n]}$ taking $U$ to $\Delta(U)$.
Applying Lemma~\ref{lem:triv automorphism} and the fact that $T,U$ are
in the same component of $\msg^{[2,n]}$, we see
$\Delta(T)=h_{n-4}\circ \Delta(U)$.


The next step is to replace the original pair of vertices connecting
$\mc$ and $\md$ by another such pair for which we can determine the
shape of $T$ and $U$ from $\Delta(T)$ and $\Delta(U)$ via jeu de
taquin more explicitly.  By Lemma~\ref{lem: correct components}, we can
consider any $T' \in V(\mc)$ that results from moving the values
$[1,n-2]$ in $T$ such that $h_{n-3}$ acts on $T'$ by switching $n-1$
and $n$. For each such $T'$, let $U'$ be the tableau representing a
vertex in $\md$ such that $(T',U')\in E_{n-3}$.  It suffices to show
that $h_{n-3}(T')=U'$ for any pair $(T',U')$  assuming that 

\begin{enumerate}
\item $T'|_{[n-4]}=U'|_{[n-4]}$.
\item $\Delta(T')=h_{n-4}(\Delta(U'))$.
\item $h_{n-3}$ acts on $T'$ and $U'$ by switching $n-1$ and $n$.
\end{enumerate}

We proceed by considering cases depending on the shape of
$T|_{[n-4]}$.  First, consider the case where $T|_{[n-4]}=U|_{[n-4]}$
has at least two northeast corners $c_1$ and $c_2$.  Assume $c_1$ is
in a higher row than $c_2$.  By rearranging the values in $[n-4]$, we
may then find $T'_1$ and $T'_2$ in $V(\mc)$ satisfying the three
assumptions above such that applying $\Delta$ to each proceeds through
$c_1$ and $c_2$, respectively.  In the jeu de taquin process on $T_1$,
all rows strictly below $c_1$ are fixed once the slide reaches $c_1$.
Similarly, all columns strictly to the left of $c_2$ are fixed once
the slide reaches $c_2$.  These two regions cover the entire shape of
$T|_{[n-4]}$, but it might not cover the whole shape of $T$.  By
assumption (3), $n-1,n$ must be in different corners of $T$.  Thus, it
can be observed that $T|_{[n-3,n]}=T_1|_{[n-3,n]}=T_2|_{[n-3,n]}$ is
completely determined by their placement in $\Delta(T_1)$ and
$\Delta(T_2)$.  Similarly, $h_{n-3}(U)$ satisfies the same assumptions
as $T$, which was uniquely determined, so $T=h_{n-3}(U)$.

Assume next that $T|_{[n-4]}$ has exactly one northeast corner $c$.
In particular, the jeu de taquin process of applying $\Delta$ to $T$
must proceed through this corner.  If $c$ is also a corner of both
$\lambda$ and $\mu$, then the jeu de taquin process does not affect the
cells containing $[n-3,n]$ in either $T$ or $U$ so
$\Delta(T)=h_{n-4}(\Delta(U))$ implies $T=h_{n-3}(U)$.

Say $c$ is in row $i$, column $j$ of $T$.  If $i>3 $ or $j-i>3$, then
$c$ is on the northeast boundary of both $T$ and
$U$.  Here we have 
used the fact that $h_{n-3}$ swaps $n-1$ and $n$ in $T$ and $U$ to ensure that 
the values in $[n-3,n]$ are not in a single row or column. Now consider the 
jeu de taquin process for $\Delta$, which must proceed through $c$. Since $c$ is
on the northeast boundary, all remaining slides are either all to the
left or all down depending only on whether or not $c$ is a northern
boundary cell or an eastern boundary cell, respectively.  Thus,
$\Delta(T)$ determines $T|_{[n-3,n]}$ and $\Delta(U)$
determines $U|_{[n-3,n]}$ where $U$.  Hence, $T=h_{n-3}(U)$.  

There are a finite number of standard shifted tableaux $T$ satisfying
the assumptions such that $T|_{[n-4]}$ has a unique corner cell in
position $(j,i)$ such that $i\leq 3$ and $j-i\leq 3$.  We leave it to the reader to
check in each case that if $T$ and $U$ exist satisfying the three
assumptions plus they have at least 9 cells, then by rearranging the
values $[n-2]$ one can find $T$ and $U$ of the same shape and
satisfying the same assumptions such that $T|_{[n-3,n]}$ and
$U|_{[n-3,n]}$ are completely determined by those assumptions and
$T=h_{n-3}(U)$.  For example, in Figure~\ref{fig: eleven cells A} we
show two possible tableaux $S$ and $T$ with different shapes such that
$\Delta(S)=\Delta(T)$.  We also show two tableaux $S'$ and $T'$ that
also satisfy the three assumptions, have the same shapes as $S$ and
$T$ respectively, but the last jeu de taquin slide ends in a different
corner.  Therefore, $S_{[8,11]}$ can be recovered from knowing both
$\Delta (S)$ and $\Delta (S')$, and similarly for $T_{[8,11]}$.
\end{proof}

\begin{figure}[h]
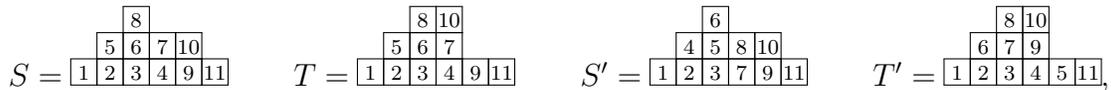

 \ytableausetup{aligntableaux=bottom}
$S=\begin{ytableau}
\none&\none&8\\
\none&5&6&7&10\\
1&2&3&4&9&11
\end{ytableau}   \hspace{.3in}
\; 
T=\begin{ytableau}
\none&\none&8&10\\
\none&5&6&7\\
1&2&3&4&9&11
\end{ytableau}
   \hspace{.3in}
\; 
S'=\begin{ytableau}
\none&\none&6\\
\none&4&5&8&10\\
1&2&3&7&9&11
\end{ytableau}   \hspace{.3in}
\; 
T'=\begin{ytableau}
\none&\none&8&10\\
\none&6&7&9\\
1&2&3&4&5&11
\end{ytableau}$,  
\caption{Example of verification left to the reader.\label{fig: eleven cells A}}
\end{figure}



\begin{remark}
For $n=8$, we may not be able to uniquely determine $U$ in the
proof above.  See Figure~\ref{fig: nine cells}.
\end{remark}

\begin{figure}[h]
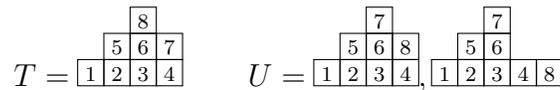

 \ytableausetup{aligntableaux=bottom}
$T=
\begin{ytableau}
\none&\none&8\\
\none&5&6&7\\
1&2&3&4
\end{ytableau}   \hspace{.3in}
\; U=
\begin{ytableau}
\none&\none&7\\
\none&5&6&8\\
1&2&3&4
\end{ytableau},  
\begin{ytableau}
\none&\none&7\\
\none&5&6\\
1&2&3&4&8
\end{ytableau}$
\caption{An example with $n=8$ and two distinct possibilities of $U$ from the proof of Lemma~\ref{lem: complete}.\label{fig: nine cells}}
\end{figure}

The following lemma is equivalent to Axiom 6 in Assaf's rules for dual equivalence graphs.

\begin{lem}\label{lem: axiom6}
  Let $\mg=(V,\sigma,E_1 \cup \ldots \cup E_{n-3})$ be a connected
signed colored graph of shifted degree $n> 9$ satisfying the Commuting Property
 such that $\mg|_{[n-1]}$ is an SDEG, and $\mg|_{[2,n]}$ is an SDEG.
Then each pair of distinct components of $\mg^{[n-1]}$ is connected by
an $(n-3)$-edge.
\end{lem} 

This proof is similar to Theorem~3.17 in \cite[p.413]{roberts2013dual}
and Lemma~\ref{lem: complete} so we only sketch it here.

\begin{proof}
The statement in the lemma is equivalent to saying that if a
component $\ma$ is connected by $(n-3)$-edges to components $\mb$ and
$\mc$, then $\mb$ and $\mc$ are connected to each other by an
$(n-3)$-edge. Using Lemma~\ref{lem: permuting ribbons}, we may apply
properties of jeu de taquin to show that this must be the case so long
as $\lambda$ is not a pyramid or $\lambda$ has more than three
northeast corners. The largest example of a shifted shape that
violates these two rules is the pyramid $(5,3,1)$ with nine cells. By
assumption, $n>9$, and so the argument is complete.
\end{proof}

\begin{proof}[\bf{Proof of Theorem~\ref{SDEG properties}}]

The fact that $\msg_\lambda$ satisfies the Commuting Property and the
Locally Standard property is proved in Lemma~\ref{lem:properties hold
for SDEGS}.  To prove the converse, assume $\mg$ is a signed colored
graph with shifted degree $n$ satisfying both of these properties.
Proceed by induction on $n$. For $n\leq 9$, the result is known by the
Locally Standard Property. We may then assume $n>9$.

By Lemma~\ref{lem: complete}, $\mg$ admits a morphism onto
$\msg_\lambda$. By Lemma~\ref{lem: injective iso}, we need only show
that this morphism is injective.  The morphism was constructed in such
a way that it is the unique extension on any component of
$\mg^{[n-1]}$ so it is injective on each component automatically.
Furthermore, the location of $n$ is constant on each component.  Let
$C$ and $D$ be two distinct components of $\mg^{[n-1]}$, and let $v
\in V(C)$ and $w\in V(D)$.  By Lemma~\ref{lem: axiom6}, there exists
an $(n-3)$-edge connecting $C$ to $D$ which necessarily moves $n$ in
the tableaux labeling its endpoints under the morphism.  Thus, the
morphism maps $v$ and $w$ to tableaux with $n$ in two different
positions.  Hence, the morphism is injective.
\end{proof}

\begin{remark}
In Theorem~\ref{SDEG properties}, $n>9$ is a sharp bound. In fact, if we consider the $n=9$ case, then there exists an infinite family of such signed colored graphs that are not SDEGs, the smallest of which is represented in Figure~\ref{fig: SDEG cover}.
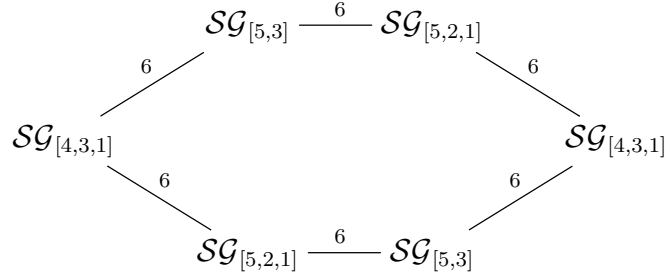
\begin{figure}[ht]
 \[
 \xymatrix{
		& \msg_{[5,3]} \ar@{-}[r]^{6} &\msg_{[5,2,1]} \ar@{-}[dr]^{6}& \\
		\msg_{[4,3,1]} \ar@{-}[ur]^{6}\ar@{-}[dr]^{6} &&&\msg_{[4,3,1]} \\
		&\msg_{[5,2,1]} \ar@{-}[r]^{6} & \msg_{[5,3]} \ar@{-}[ur]^{6} &
}
\]
\caption{A graph $\mg$ represented by the isomorphism types of components in $\mg|_{[8]}$ and the connection of these components via 6-edges. Here, $\mg$ satisfies the Commuting Property, $\mg|_{[8]}$ and $\mg|_{[2,9]}$ are SDEGs, but $\mg$ is not an SDEG.\label{fig: SDEG cover}}
\end{figure}
\end{remark}

\section{Open Problems}\label{s:open}

We conclude with some interesting open problems.  

\begin{enumerate}
\item What are the Coxeter-Knuth relations, graphs and Little bumps in
  other Coxeter group types?  Tao Kai Lam described Coxeter-Knuth
  relations in type $D$ \cite{lam1995b}.  We have not found an analog
  of the Little bump algorithm that commutes with these relations.

\item In type $A$, the simple part of every Kazhdan-Lusztig graph is a
  Coxeter-Knuth graph and vice versa as mentioned in the introduction.
  This is not true in type $B$.  What set of relations goes with the
  Kazhdan-Lusztig graphs in general?  This would also generalize the
  RSK algorithm and Knuth/DEG relations.   

\item  What is the significance of the Little bumps in Schubert calculus? 

\item What interesting symmetric functions expand as a positive sum of
  Schur Q's?  Are there natural expansions of certain symmetric
  functions first into peak quasisymmetric functions?

\item What is the diameter of the largest connected component of a
  Coxeter-Knuth graph for permutations or signed permutations of
  length $n$?

\end{enumerate}

  

\section{Acknowledgments}\label{s:ack}
Many thanks to Andrew Crites, Mark Haiman, Tao Kai Lam, Brendan
Pawlowski, Peter Winkler and an anonymous referee for helpful discussions on this work.

\end{document}